\documentclass[11pt]{amsart}
\usepackage{amsmath,color,mathrsfs,enumerate,amssymb,latexsym,amsfonts,amsbsy,amsthm,mathtools,graphicx}
\usepackage{hyperref}

\let\g=\gamma

\let\la=\lambda

\let\s=\sigma
\let\f=\frac
\let\p=\psi

\let\D=\Delta

\let\pa=\partial

\newcommand{\beq}{\begin{equation}}
	\newcommand{\eeq}{\end{equation}}
\newcommand{\ben}{\begin{eqnarray}}
	\newcommand{\een}{\end{eqnarray}}
\newcommand{\beno}{\begin{eqnarray*}}
	\newcommand{\eeno}{\end{eqnarray*}}

\newtheorem{defi}{Definition}[section]
\newtheorem{thm}{Theorem}[section]
\newtheorem{lem}{Lemma}[section]
\newtheorem{rmk}{Remark}[section]
\newtheorem{cor}{Corollary}[section]
\newtheorem{prop}{Proposition}[section]

\numberwithin{equation}{section}

\newcommand{\andf}{~\text{ and }~}
\newcommand{\with}{~\text{ with }~}

\def\eqdef{\buildrel\hbox{\footnotesize def}\over =}

\def\dive{{\mathop{\rm div}\nolimits}\,}

\def\Supp{\mathop{\rm Supp}\nolimits\ }
\def\DD{\dot{\Delta}}

\let\f=\frac
\let\p=\partial

\def\N{\mathop{\mathbb  N\kern 0pt}\nolimits}
\def\Q{\mathop{\mathbb  Q\kern 0pt}\nolimits}
\def\R{{\mathop{\mathbb R\kern 0pt}\nolimits}}
\def\Z{\mathop{\mathbb  Z\kern 0pt}\nolimits}
\def\PP{\mathop{\mathbb P\kern 0pt}\nolimits}
\def\dis{\displaystyle}

\def\tp{\tilde{p}}
\def\tq{\tilde{q}}
\def\bq{\bar{q}}

\def\dx{\mathrm{d}x}
\def\dy{\mathrm{d}y}

\def\dt{\mathrm{d}t}
\def\ds{\mathrm{d}s}

\newcommand{\BP}{\dot{B}^{-1+\f3p}_{p,\infty}}
\newcommand{\BPS}{\dot{B}^{-1+\f{3}{\tilde{p}}+\sigma}_{\tilde{p},\infty}}

\newcommand{\LQQ}{L^{q,\infty}}

\newcommand{\BB}{\dot{B}^{\f12}_{2,\infty}}

\allowdisplaybreaks

\begin{document}
%\title[Cannone-Meyer-Planchon solutions of 3-D Inhomogeneous Navier-Stokes system]
\title[The inhomogeneous Navier-Stokes system]
{Global well-posedness and self-similar solution of the inhomogeneous Navier-Stokes system}

\author[T. Hao]{Tiantian Hao}
\address[T. Hao]{School of Mathematical Sciences, Peking University, Beijing 100871, China}
\email{haotiantian@pku.edu.cn}

\author[F. Shao]{Feng Shao}
\address[F. Shao]{School of Mathematical Sciences, Peking University, Beijing 100871,  China}
\email{fshao@stu.pku.edu.cn}

\author[D. Wei]{Dongyi Wei}
\address[D. Wei]{School of Mathematical Sciences, Peking University, Beijing 100871, China}
\email{jnwdyi@pku.edu.cn}

\author[P. Zhang]{Ping Zhang}
\address[P. Zhang]{Academy of Mathematics \& Systems Science and Hua Loo-Keng Center for Mathematical Sciences, Chinese Academy of Sciences, Beijing 100190, China, and School of Mathematical Sciences, University of Chinese Academy of Sciences, Beijing 100049, China}
\email{zp@amss.ac.cn}

\author[Z. Zhang]{Zhifei Zhang}
\address[Z. Zhang]{School of Mathematical Sciences, Peking University, Beijing 100871, China}
\email{zfzhang@math.pku.edu.cn}

\date{\today}

\begin{abstract}
    In this paper, we study the global well-posedness of the 3-D inhomogeneous  incompressible Navier-Stokes system (INS in short) with initial density $\rho_0$ being discontinuous and initial velocity $u_0$ belonging to some critical space.  Firstly, if $\rho_0u_0$ is sufficiently small in the space $\BP(\R^3)$ and $\rho_0$ is close enough to a positive constant in $L^\infty$, we  establish the global existence of strong solution to (INS) for $3<p<\infty$ and provide the uniqueness of the solution for $3<p<6$. This result corresponds to Cannone-Meyer-Planchon solution of the classical Navier-Stokes system. Furthermore, with the additional assumption that $u_0\in L^2(\R^3)$, we prove the weak-strong uniqueness between Cannone-Meyer-Planchon solution and Lions weak solution of (INS).  Finally, we prove the global well-posedness of (INS) with $u_0\in \BB(\R^3)$ being small and only an upper bound on the density.
 This gives the first existence result of the forward self-similar solution for (INS).
\end{abstract}
\maketitle

\section{Introduction}
In this paper, we consider the following inhomogeneous incompressible Navier-Stokes system in $ \mathbb{R}^{+}\times\R^3$
\begin{equation}\label{INS}
	\left\{
	\begin{array}{l}
		\partial_t\rho+u\cdot\nabla \rho=0,\\
		\rho(\partial_tu+u\cdot\nabla u)-\Delta u+\nabla P=0,\\
		\dive u = 0,\\
		(\rho, \PP(\rho u))|_{t=0} =(\rho_{0}, m_0),
	\end{array}
	\right.
\end{equation}
where $\rho,~u$ stand for the density and velocity of the fluid respectively, and $P$ is a scalar pressure function, $\PP$ is the Leray projector into divergence-free vector fields.  This system is known as a model for the evolution of a multi-phase flow consisting of several immiscible, incompressible fluids with different densities. We refer to \cite{Lions_vol1} for a detailed derivation of \eqref{INS}.

 When $\rho_0\equiv 1$, the system \eqref{INS} is reduced to the classical Navier-Stokes system (NS in short). Similar to (NS), the inhomogeneous Navier-Stokes system \eqref{INS} also has a scaling invariance.
More precisely, if $(\rho,u,P)$ is a solution of \eqref{INS} on $[0, T]\times\R^d$, then for all $\la>0$, the triplet $(\rho_\la, u_\la, P_\la)$ defined by
    \begin{equation}\label{Eq.scaling}
        (\rho_\la, u_\la, P_\la)(t,x)\eqdef\left(\rho(\la^2t,\la x), \la u(\la^2t,\la x), \la^2P(\la^2t,\la x)\right)
    \end{equation}
    solves \eqref{INS} on $[0, T/\la^2]\times\R^d$. In this paper, we are concerned with the well-posedness of \eqref{INS} in the so-called {\it critical functional framework}, which is to say, in functional spaces with scaling invariant norms.

%Lady\v{z}enskaja and Solonnikov \cite{LS} first addressed the question of unique solvability of \eqref{INS} in a bounded domain $\Omega$. Under the assumption that $u_0\in W^{2-\frac{2}{p},p}(\Omega)(p>d)$ is divergence free and vanishes on $\pa\Omega$ and  $\rho_0\in C^1(\Omega)$ is bounded away from zero, they proved the global well-posedness in dimension $d=2$, and local well-posedness in dimension $d=3$.

Let us first recall  some solutions of  3-D (NS) in critical spaces: Fujita-Kato solution in $\dot H^{1/2}$ \cite{Fujita-Kato}, Kato solution in $L^3$ \cite{Kato}, Cannone-Meyer-Planchon solution in $\dot B^{-1+3/p}_{p,\infty}$ for $p\in(3,\infty)$ \cite{Cannone-Meyer-Planchon}, Koch-Tataru solution in $\text{BMO}^{-1}$ \cite{Koch-Tataru}, etc. Recently, several works are devoted to the extension of (NS) to (INS) in the critical functional framework. Danchin \cite{Danchin2003} and Abidi \cite{Abidi2007} proved that if $\rho_0$ is close to a positive constant in $\dot B^{d/p}_{p,1}(\R^d)$ and $u_0$ is sufficiently small in $\dot B^{-1+d/p}_{p,1}$, then there is a global solution to \eqref{INS} with the initial data $(\rho_0, u_0)$ for all $p\in(1,2d)$ and the uniqueness holds for $p\in(1,d]$. The existence result has been extended to general Besov spaces in \cite{Abidi-Paicu2007, Danchin_Mucha2012, Paicu_zhang2012} even without the size restriction for the density, see \cite{Abidi_Gui2021, Abidi_Gui_Zhang2012, Abidi_Gui_Zhang2013, Abidi_Gui_Zhang2023}. In \cite{Danchin_Mucha2012}, Danchin and Mucha proved the existence and uniqueness for $p\in[1,2d)$, with $\rho_0$ close to a positive constant in the multiplier space $\mathcal M(\dot B^{-1+d/p}_{p,1})$.

In all these aforementioned results, the density has to be at least continuous or near a positive constant, which excludes some physical cases when the density is discontinuous and has large variations along a hypersurface (but still bounded). For results concerning initial data $u_0\in H^s(\R^d)$ with $s>\f d2-1$, and $0\leq\rho_0\in L^{\infty}$, bounded away from zero, we refer to \cite{K, S, DM3, Paicu_Zhang_Zhang_CPDE, CZZ}.
Note that the norms of the initial velocity are not critical in the sense that they are not invariant under the scaling \eqref{Eq.scaling}.

The first result where $\rho_0$ is merely bounded and $u_0$ lies in a critical space was obtained by the fourth author \cite{Zhang_Adv}, where he established the global existence of solutions to 3-D (INS) with initial density $0\leq\rho_0\in L^\infty(\R^3)$, bounded away from zero, and initial velocity $u_0$ sufficiently small in a critical Besov space $\dot B^{1/2}_{2,1}(\R^3)$. Later on, the uniqueness has been proved by Danchin and Wang \cite{DW}. In \cite{DW}, the authors proved the global existence of solutions to (INS) with $\|\rho_0-1\|_{L^\infty(\R^3)}$  and $\|u_0\|_{ \dot B^{-1+3/p}_{p,1}(\R^3)}$ being small for $p\in (1,3)$, and also uniqueness for $p\in (1,2]$; furthermore, weak-strong uniqueness is proved for $p\in (1,3)$ under the additional assumption that $u_0\in L^2(\R^3)$. In these results, the summability index 1 in the Besov space ensures that the gradient of the velocity is in $L^1(0,T; L^\infty)$, which is a key ingredient in the proof of uniqueness, see also \cite{Danchin2024}.

In \cite{HSWZ2}, the authors have extended the 2-D Leray  weak solution and Fujita-Kato solution to (INS),  where the density is not required to be close to a constant and even allows the presence of vacuum in the case of $\R^3$. The proof of uniqueness is rather non-trivial due to the absence of the condition $\nabla u\in L^1(0,T;L^{\infty})$ when $u_0\in L^2(\R^2)$ or $u_0\in \dot H^{\f12}(\R^3)$. We also mention \cite{DV} where Danchin and Vasilyev obtained a well-posedness result for (INS) in critical tent spaces where the initial velocity lies in a subspace of $\mathrm{BMO}^{-1}(\R^3)$.\smallskip

\subsection{Main results}
Let us first recall a celebrated theorem by Cannone, Meyer and Planchon.

\begin{thm}[\cite{Cannone-Meyer-Planchon}]\label{thm.Cannone-Meyer-Planchon}
    Let $p\in (3,\infty)$. Given a divergence-free vector field $u_0\in \dot{B}^{-1+\f3p}_{p,\infty}(\R^3)$ small enough,
    then (NS) has a unique global solution $u\in L^{\infty}(\R^+;\dot{B}^{-1+\f3p}_{p,\infty}(\R^3))\cap \widetilde{L}^1(\R^+;\dot{B}^{1+\f3p}_{p,\infty}(\R^3)).$
\end{thm}

Our first main result is to extend Cannone-Meyer-Planchon solution of (NS) to (INS).
\begin{thm}\label{thm.INS_Cannone-Meyer-Planchon}
    Let $p\in(3,\infty)$ and $q\in (1,\infty)$ be such that $\f3p+\f2q=1$. There exists a constant $\varepsilon_0>0$ such that if
    \begin{align}\label{thm.small condition}
        \|\rho_0-1\|_{L^{\infty}(\R^3)}+\|m_0\|_{\BP(\R^3)}<\varepsilon_0,
    \end{align}
    then the system \eqref{INS} has a global strong solution $(\rho,u,\nabla P)$ with $ \|\rho(t)-1\|_{L^{\infty}}=\|\rho_0-1\|_{L^{\infty}}$,  and the following properties:
    \begin{itemize}
        \item $\rho u, t\Delta u, t\nabla P\in \LQQ(\R^+; L^p)$ and $t^{1+\f{1}{2q}}\Delta u,t^{1+\f{1}{2q}}\nabla P\in L^{2q,\infty}(\R^+; L^p)$;
        \item $t u,t^2\dot{u}\in L^\infty(\R^+;\dot B^{1+\f3p}_{p,\infty})$ and $t^{1+\f{1}{2q}}u\in L^\infty(\R^+;\dot{B}^{1+\f3p+\f1q}_{p,\infty})$, {where $\dot u\eqdef u_t+u\cdot\nabla u$};
        \item $t\dot{u}, t^2\Delta\dot{u}\in \LQQ(\R^+; L^p)$;
        \item $\|u(t)\|_{L^p}\leq C \varepsilon_0 t^{-\f1q}$, $\|\nabla u(t)\|_{L^p}\leq C \varepsilon_0 t^{-\f12-\f1q}$;
        \item $\|u(t)\|_{L^{\infty}}\leq C \varepsilon_0 t^{-\f12}$, $\|\nabla u(t)\|_{L^{\infty}}\leq C \varepsilon_0 t^{-1}$, where $C$ only depends on $\|\rho_0\|_{L^{\infty}}$.
    \end{itemize}
    Moreover, thus obtained solution is unique for $3<p<6$.
\end{thm}

\begin{rmk}\label{cor.high oscillation}
As $-1+\f3p<0$, our result generates the global solution of (INS) with  highly oscillating initial velocity. Let $\varepsilon_0>0$ be given by Theorem \ref{thm.INS_Cannone-Meyer-Planchon} and let $\rho_0\in L^\infty(\R^3)$ be such that $\|\rho_0-1\|_{L^{\infty}}<\varepsilon_0$. Given $\phi\in \mathcal{S}(\R^3)$, we define
 \begin{align*}
     \varphi_{\varepsilon}(x_1, x_2, x_3)\eqdef\cos\bigl(\f{x_3}{\varepsilon}\bigr)\phi(x_1,x_2,x_3),\quad\forall\ \varepsilon\in(0,1).
 \end{align*}
 There is a constant $C>0$ such that for $\varepsilon$ small enough, the smooth divergence free vector field
 \begin{align*}
     u_{0,\varepsilon}(x)=(\p_2\varphi_{\varepsilon}(x),-\p_1\varphi_{\varepsilon}(x),0)
 \end{align*}
 satisfies (using the embedding $L^3(\R^3)\hookrightarrow\BP(\R^3)$ and the uniform boundedness of $u_{0,\varepsilon}$ in $L^3$)
 \begin{align*}
\|u_{0,\varepsilon}\|_{\BP}+\|\rho_0u_{0,\varepsilon}\|_{\BP}\leq C(\varepsilon^{1-\f3p}+\varepsilon_0).
 \end{align*}
Then for $\varepsilon$ small enough, $(\rho_0,u_{0,\varepsilon})$ generates a unique global solution to \eqref{INS} for $p\in (3,6)$.
\end{rmk}

In a seminal paper \cite{Leray1934}, Leray constructed a global weak  solution of classical Navier-Stokes system for any initial data in $L^2(\R^3)$.
%There are a number of papers related to the weak-strong uniqueness of the Navier-Stokes system. It was proven in \cite{Leray1934} that for $u_0\in H^1(\R^3)$ and $u_0\in L^2\cap L^p(\R^3)(3<p\leq \infty)$, they proved the uniqueness in the slightly narrower class of Leray's solutions.
In \cite{Chemin2011}, Chemin proved the weak-strong uniqueness of Cannone-Meyer-Planchon solution with $u_0\in H^s\cap \dot{B}^{-1+\f3p}_{p,\infty}(\R^3)$ for $s>0$, which was later refined by Barker to $u_0\in L^2\cap \dot{B}^{-1+\f3p}_{p,\infty}(\R^3)$  \cite{Barker2018}. Dong and Zhang \cite{DZ} provided the weak-strong uniqueness of Koch-Tataru solution with $u_0\in H^s\cap \text{BMO}^{-1}(\R^3)$ for $s>0$.
For a more comprehensive historical overview of weak-strong uniqueness, readers may refer to \cite{Germain2006}.

Lions \cite{Lions_vol1} provided a global weak solution to \eqref{INS} for initial velocity in $L^2(\R^3)$ and initial density satisfying $0\leq\rho_0\in L^{\infty}$ along with some additional technical assumptions.
Recently, in \cite{HSWZ1}, the authors prove the 2-D weak-strong uniqueness  result, which includes the case of density patch or vacuum bubble, while the initial velocity $u_0$ belongs to non-critical space $H^1(\R^2)$. In a critical framework, the weak-strong uniqueness has been obtained in \cite{DW, CSV, HSWZ2}. In  \cite{HSWZ2}, the authors proved the weak-strong uniqueness under the following assumptions:
\begin{itemize}
    \item 2-D case: $u_0\in L^2(\R^2)$, $\rho_0\in L^\infty(\R^2)$ and $\inf\rho_0>0$;
    \item 3-D case: $u_0\in L^2(\R^3)\cap \dot H^{1/2}(\R^3)$ and $0\leq\rho_0\in L^\infty(\R^3)$.
\end{itemize}

For Cannone-Meyer-Planchon solution of \eqref{INS} obtained in Theorem \ref{thm.INS_Cannone-Meyer-Planchon}, we establish the following weak-strong uniqueness.

\begin{thm}[Weak-strong uniqueness]\label{thm.weak-strong uniqueness}
Assume that the initial data satisfy $\rho_0u_0\in \BP(\R^3) \cap L^2(\R^3)$ for $p\in (3,\infty)$ and the assumption \eqref{thm.small condition}.
Let $(\rho,  u)$ be the solution of \eqref{INS} obtained in Theorem \ref{thm.INS_Cannone-Meyer-Planchon}. Then $(\rho, u)$ is unique among Lions weak solutions associated with $(\rho_0, u_0)$.
\end{thm}

\begin{rmk}
The uniqueness part in the critical spaces in Theorem \ref{thm.INS_Cannone-Meyer-Planchon} requires $p\in(3,6)$. In contrast, the weak-strong uniqueness in Theorem \ref{thm.weak-strong uniqueness} holds for all $p\in (3,\infty)$. However, due to technical reasons, our method for proving Theorem \ref{thm.weak-strong uniqueness}, which relies on a decomposition of $\BP$, can not be extended to the case $p=+\infty$. Consequently, our method can not applicable to $\mathrm{BMO}^{-1}(\R^3)$.
\end{rmk}

Very recently, the authors \cite{HSWZ2} have established global well-posedness of the system \eqref{INS} with $u_0\in \dot{H}^{1/2}(\R^3)$ being small and an upper bound on $\rho_0$. Building upon that, we further improve the condition from $\dot{H}^{1/2}$ to $\dot B^{1/2}_{2,\infty}$, and our third result is stated as follows {(here the initial condition is $ (\rho,u)|_{t=0}=(\rho_0,u_0)$)}.

\begin{thm}\label{thm.self similar}
 Given the initial data $(\rho_0,u_0)$ satisfying $0\leq \rho_0(x)\leq \|\rho_0\|_{L^{\infty}}$ and $\rho_0\not\equiv 0$, $ u_0\in \BB(\R^3)$, $\dive u_0=0$, there exists $\varepsilon_0>0$ depending only on $\|\rho_0\|_{L^{\infty}}$ such that if $\|u_0\|_{\BB}< \varepsilon_0$, then the system \eqref{INS} has a unique global weak solution $(\rho,u,\nabla P)$ with $0\leq \rho(t,x)\leq \|\rho_0\|_{L^{\infty}}$, $\sqrt\rho u\in C([0, +\infty); L^{3,\infty}(\R^3))$ and the following properties (for any $0<T_1<\infty$)
    \begin{itemize}
        \item $t^{1/4}\nabla u\in L^\infty(\R^+; L^2(\R^3))$;
        \item $\sqrt\rho\dot u,\ \nabla^2u,\ \nabla\dot u\in L^2(T_1,2T_1; L^2(\R^3))$, and $\sqrt\rho\dot u\in L^\infty(T_1,2T_1; L^2(\R^3))$, where $\dot u\eqdef u_t+u\cdot\nabla u$;
        \item $\sqrt\rho u_t,\ \sqrt{\rho_0} u_t,\ \nabla u_t\in L^2(T_1,2T_1; L^2(\R^3))$ and $\sqrt\rho u_t\in L^\infty(T_1,2T_1; L^2(\R^3))$;
        \item $\nabla u\in L^1(T_1,2T_1; L^\infty(\R^3))$.
    \end{itemize}
\end{thm}

\begin{rmk}
In contrast to the $\dot{H}^{1/2}$ initial data discussed in \cite{HSWZ2}, Theorem \ref{thm.self similar} gives the first existence result of the forward self-similar solution for (INS).
For instance, consider $u_0(x)=\f{\varepsilon_0}{|x|}\bigl(\f{-x_2}{|x|},\f{x_1}{|x|},0\bigr)$, it is straightforward to verify that $\dive u_0=0$, $\|u_0\|_{\BB}\leq C\varepsilon_0$, but $u_0\not\in \dot{H}^{1/2}$.
\end{rmk}

%Here, we provide a summary remark regarding the uniqueness of the aforementioned theorems.
%\begin{rmk}
%    \begin{itemize}
%        \item[(1)] Regarding our initial data assumptions, we can not derive $\nabla u \in L^1_T(L^\infty)$, which is crucial for controlling  $\delta\!\rho$.
%        Therefore, we primarily rely on the novel method proposed in \cite{HSWZ2}, which does not require $\nabla u \in L^1_t(L^{\infty})$.
%
%        \item[(2)] For the uniqueness part of Theorem \ref{thm.INS_Cannone-Meyer-Planchon}, \ref{thm.weak-strong uniqueness} and \ref{thm.self similar}, we note that the third Besov index being infinity, which especially includes forward self-similar solutions, makes the smallness condition of the initial velocity necessary in all cases.
%
%        \item[(3)] In the uniqueness proof of Theorem \ref{thm.self similar}, due to technical reasons, we only give the uniqueness of strong solutions.
%        The weak-strong uniqueness remains open.
%        \end{itemize}
%
%\end{rmk}

\subsection{Difficulties and our ideas}

 For $m_0\in \BP$ with $p\in (3,\infty)$, most of the usual product laws cannot be used in this case. Our proof is based on an important observation (see Proposition \ref{prop.u_L^q,infty(L^p)}): by considering the equation of $\PP(\rho u)$ and using the assumption that $\|\rho_0-1\|_{L^{\infty}}$ and $\|m_0\|_{\BP}$ are small, we first derive the critical estimate $\|u\|_{L^{q,\infty}_T(L^p)}$ for $p, q$ satisfying $3/p+2/q=1$. With this estimate in hand, the other necessary estimates can be derived by using maximal regularity estimates for Stokes equations (Lemma \ref{lem.maximal regularity}). For the uniqueness part, we denote $\delta\!\rho\eqdef \rho-\bar{\rho},~ \delta\!u\eqdef u-\bar{u}$.
We estimate $\delta\!\rho$ by using the method developed in \cite{HSWZ2}, which does not require $\nabla u\in L^1_T(L^{\infty})$. We note that this method is also used in the weak-strong uniqueness of Theorem \ref{thm.weak-strong uniqueness}, and the uniqueness part in Theorem \ref{thm.self similar}. For $\delta\!u$, we consider the equation of $\PP(\rho \delta\!u)$, and also use the smallness condition of $\|m_0\|_{\BP}$ and $\|\rho_0-1\|_{L^{\infty}}$ to complete the proof of uniqueness.

To prove the weak-strong uniqueness, we mainly rely on the decomposition of $\BP$ proposed by Barker in \cite{Barker2018}. Specifically, we can decompose $m_0$ into $(m_0)_1$ and $(m_0)_2$, which are all divergence-free and belong to $\BPS(\sigma>0$ small)
and $\dot{B}^s_{2,\infty}(0<s<1/2)$ respectively. We set $m_{0,N}$ as the sum of $(m_0)_1$ and the low-frequency of $(m_0)_2$, and $u_{0,N}$ as the corresponding initial velocity with $\PP(\rho_0u_{0,N})\in \BP\cap\BPS$. Then for the equation \eqref{INS_N} associated with $(\rho_0,u_{0,N})$, we have the critical estimates (uniformly in $N$) presented in Section \ref{Sec.Cannone-Meyer-Planchon}, and the subcritical estimates provided in Section \ref{Sec.weak-strong}.  We denote $\delta\!u^{(N)}\eqdef u-u^{(N)},~\delta\!\rho^{(N)}\eqdef \rho-\rho^{(N)}$. From these, we can conclude that {(see \eqref{minus.u_0}, \eqref{N2T}, \eqref{uniqueness2.Gronwall1}, \eqref{uN} and \eqref{Eq.uniqueness} for more details)}
\begin{align*}
&\|\sqrt\rho\delta\!u^{(N)}\|_{L^\infty(0, T; L^2)}\leq C\|u_0-u_{0,N}\|_{L^2}\exp{\int_0^TC\bigl(\|\nabla {u}^{(N)}\|_{L^{\infty}}+\|t\nabla\dot{u}^{(N)}\|_{L^{\infty}}+\|t\dot{u}^{(N)}\|^2_{L^{\infty}})\,\dt},\\
    &\|u_0-u_{0,N}\|_{L^2}\leq C2^{-Ns},\quad \int_0^T(\|\nabla u^{(N)}\|_{L^{\infty}}+\|t\nabla\dot{u}^{(N)}\|_{L^{\infty}}+\|t\dot{u}^{(N)}\|^2_{L^{\infty}})\,\dt\leq C\varepsilon_0^{\f{p}{\tilde{p}}}(\ln T+N),
\end{align*}
which imply that
\begin{align*}
\|\sqrt\rho\delta\!u^{(N)}\|_{L^{\infty}(0,T;L^2)}+\|\nabla \delta\!u^{(N)}\|_{L^2(0,T;L^2)}\leq CT^{C\varepsilon_0^{\f{p}{\tilde{p}}}} 2^{-N(s-C\varepsilon_0^{\f{p}{\tilde{p}}})}\to 0 \,\text{~as~}\,N\to \infty.
\end{align*}
Therefore, we can achieve that  any weak solution provided by Lions with initial data $(\rho_0,u_0)$ coincides with the limit of $(\rho^{(N)},u^{(N)})$. This implies the weak-strong uniqueness.

For $u_0\in\BB$, we note that the third Besov index is infinity, which prevents us from obtaining time integration over the whole $\R^+$. Consequently, we can only estimate $u$ in the form $\|u\|_{L^q(T_1,T_2;L^p)}$ for any $0<T_1<T_2<\infty$. Our proof relies on a crucial observation (see Lemma \ref{lem.rho u L3}): by analyzing the equation of $u_j$ and applying high-low frequency decomposition, we can derive the critical estimates $\|\sqrt{\rho}u\|_{L^{\infty}(\R^+;L^{3,\infty})}$ and $\|t^{1/4}\nabla u\|_{L^{\infty}(\R^+;L^2)}$. Then other estimates can be obtained by standard argument.
For the uniqueness part, we have already discussed the difficulties associated with $\delta\!\rho$. As for $\delta\!u$, our method relies on the smallness assumption of $u_0$ and an important inequality: $\|\sqrt{\rho}\delta\!u\|_{L^{\infty}(0,t;L^2)}+\|\nabla\delta\!u\|_{L^2(0,t;L^2)}\leq Ct^{1/4}$. This  inequality ensures
\begin{align*}
\sup_{t\in [\delta,T]} \|\sqrt\rho\delta\!u(t)\|^2_{L^2}+\int_{\delta}^T\|\nabla \delta\!u(t)\|^2_{L^2}\,\dt\leq C
\|\sqrt\rho\delta\!u(\delta)\|^2_{L^2} \exp \bigl(C\int_{\delta}^T \gamma(t)\,\dt\bigr),
\end{align*}
where $\gamma(t)=\|\nabla \bar u(t)\|_{L^2}^4+t^{3/2}\|\nabla\dot{\bar u}(t)\|_{L^2}^2$. Using the estimates in Lemma \ref{lem.rho u L3} and Lemma \ref{lem.nabla u L_infty},  we can conclude that
\begin{align*}
\sup_{t\in [\delta,T]}\|\sqrt\rho\delta\!u(t)\|^2_{L^2}+\int_{\delta}^T\|\nabla \delta\!u(t)\|^2_{L^2}\,\ds\leq C\delta^{1/2} (T/\delta)^{C\varepsilon_0^2} \to 0  \,\text{~as~} \,\delta\to 0,    \end{align*}
which implies that $\delta\!u(t)=0$ on $[0,T]$.

\subsection{Notations}

\begin{itemize}
    \item $\R^+{\eqdef}(0, +\infty)$. $D_t\eqdef{(\pa_t+u\cdot \nabla)}$ is the material derivative.
    \item For a Banach space $X$ and an interval $I\subset\R$, we denote by $C(I;X)$ the set of continuous functions on $I$ with values in $X$. For $p\in[1,+\infty]$, the notation $L^p(I; X)$ stands for the collection of measurable functions on $I$ with values in $X$, such that $t\mapsto\|f(t)\|_X$ belongs to $L^p(I)$. For any $T>0$, we abbreviate $L^p((0, T); X)$ to $L^p(0, T; X)$ and sometimes we further abbreviate to $L^p(0, T)$ if there is no confusion. We also write $L_t^qL^p=L^q(0, t; L^p(\R^3))$
    \item $B_R{\eqdef}B(0,R), ~\forall~ R>0$.
    \item For $s\in\R$ and $p\in[1,+\infty]$, we denote by $\dot W^{s,p}(\R^d)$ (or shortly $\dot W^{s,p}$) the standard homogeneous Sobolev spaces, and we also denote $\dot H^s{\eqdef}\dot W^{s,2}$.
    \item The definitions of Besov spaces $\dot{B}^s_{p,r}$ will be recalled in Appendix \ref{Appen_proof}.
    \item Throughout the text, $A\lesssim B$ means that $A\leq CB$, and we shall always denote $C$ to be a positive absolute constant which may vary from line to line. The dependence of the constant $C$ will be explicitly indicated if there are any exceptions.
\end{itemize}

The rest of the paper is organized as follows. In Section \ref{Sec.Cannone-Meyer-Planchon}, we prove Theorem \ref{thm.INS_Cannone-Meyer-Planchon}; with the critical estimates in Section \ref{Sec.Cannone-Meyer-Planchon} and the subcritical estimates in Section \ref{Sec.weak-strong}, we prove Theorem \ref{thm.weak-strong uniqueness}; in Section \ref{Sec.self similar}, we prove Theorem \ref{thm.self similar}. Finally in the Appendix \ref{Appen_proof}, we collect some basic facts on Littlewood-Paley theory and Lorentz spaces.

\section{Cannone-Meyer-Planchon solutions of (INS)}\label{Sec.Cannone-Meyer-Planchon}

\subsection{A priori estimates}\label{Subsec.2.1}

In this subsection, we will give some estimates which are needed to prove Theorem \ref{thm.INS_Cannone-Meyer-Planchon}. We first recall O'Neil's convolution inequality.
\begin{lem}[Theorem 2.6, \cite{ONeil}]\label{lem.convolution inequality}
Suppose $1\leq p_1,p_2,q_1,q_2,r\leq \infty$ are such that $r<\infty$ and
\begin{align*}
1+\f1r=\f{1}{p_1}+\f{1}{p_2} \andf \f{1}{q_1}+\f{1}{q_2}\geq \f1s.
\end{align*}
Suppose that $f\in L^{p_1,q_1}(\R^d)$ and $g\in L^{p_2,q_2}(\R^d)$, then it holds that
\begin{align*}%\label{estimates.convolution}
\|f*g\|_{L^{r,s}(\R^d)}\leq 3r \|f\|_{L^{p_1,q_1}(\R^d)}\|g\|_{L^{p_2,q_2}(\R^d)}.
\end{align*}
\end{lem}

By Lemma \ref{lem.convolution inequality}, we can easily get the following estimates.

\begin{cor}\label{cor.convolution}
Let $f(t,x):[0,T)\times \R^3\to \R^3$ be a smooth function and define $Af(t,x){\eqdef}\int_0^t e^{(t-s)\Delta}$ $f(s,x)\,\ds$. Then for any $p\in (3,\infty),~q\in (1,\infty)$ satisfying $\f3p+\f2q=1$, there exists a constant $C$ independent of $T$ such that
\begin{align}\label{estimate.nabla Af}
    \|\nabla Af\|_{\LQQ(0,T;L^p)}\leq C \|f\|_{L^{\f q2,\infty}(0,T;L^{\f p2})} \andf \|\Delta Af\|_{\LQQ(0,T;L^p)}\leq C \|f\|_{L^{q,\infty}(0,T;L^{p})}.
\end{align}
Moreover, for $p\in(3,6)$ and $r\in(1,\infty)$ which satisfies $\f3p+\f2q=1,~\f1r=\f2q+\f12$, there also holds
\begin{align}\label{estimate.Af}
    \|Af\|_{\LQQ(0,T;L^p)}\leq C \|f\|_{L^{r,\infty}(0,T;L^{\f p2})}.
\end{align}
\end{cor}

\begin{proof}
The second inequality of \eqref{estimate.nabla Af} follows from [Proposition 2.1, \cite{DMT}].  As we can write
    \begin{align*}
     e^{(t-s)\Delta} f(s,x)=\f{1}{(4\pi (t-s))^{3/2}}\int_{\R^3}e^{-\f{|x-y|^2}{4(t-s)}} f(s,y)\,\dy,
    \end{align*}
by Young's inequality we have
    \begin{align*}
     \|\nabla e^{(t-s)\Delta} f\|_{L^p}\leq C|t-s|^{-\f3{2p}-\f12}\|f(s)\|_{L^{\f p2}},\quad  \| e^{(t-s)\Delta} f\|_{L^p}\leq C|t-s|^{-\f3{2p}}\|f(s)\|_{L^{\f p2}}.
      \end{align*}
  Then the first inequality of \eqref{estimate.nabla Af} and \eqref{estimate.Af} can be easily derived by adapting Lemma \ref{lem.convolution inequality}.
\end{proof}

\begin{prop}\label{prop.u_L^q,infty(L^p)}
Let $p\in (3,\infty),~q\in (1,\infty)$ satisfying $\f3p+\f2q=1$.
Let $(\rho,u)$ be a smooth solution of \eqref{INS} on $[0,T]\times \R^3$. There exists a constant $\varepsilon_0\in(0,1)$
such that if
\begin{align}\label{assumption.rho(t)}
    {\dis\sup_{t\in [0,T]}}
    \|{\rho(t)}-1\|_{L^{\infty}}+\|m_0\|_{\BP}< \varepsilon_0,
\end{align}
then there holds
\begin{align}\label{estimate.u_L^q,infty(L^p)}
    \|u\|_{\LQQ(0,T;L^p)}\leq C\|m_0\|_{\BP},
\end{align}
where $C$ is independent of $T$.
\end{prop}

\begin{proof}
Let $v\eqdef \PP(\rho u)$ with $v_0= m_0$.
By \eqref{INS}, we know that $v$ satisfies
\begin{equation*}%\label{eq.v}
     \left\{
     \begin{array}{l}
     \partial_tv-\Delta v=\Delta(u-v)-\PP\dive(\rho u\otimes u), \quad (t,x)\in \R^{+}\times\R^3\\
     \dive v=0,\\
     v|_{t=0}=v_0.
     \end{array}
     \right.
\end{equation*}
Then we can write
\begin{align}\label{eq.v_mild solution}
    v(t,x)=e^{t\Delta}v_0+\int_0^t e^{(t-s)\Delta}\Bigl(\Delta(u-v)-\PP\dive(\rho u\otimes u)\Bigr)(s,x)\,\ds.
\end{align}

Firstly, according to the equivalent definition of Besov space with negative indices provided in Definition \ref{defBesov}, we can deduce that
\begin{align}\label{estimate.heat v_0}
\|e^{t\Delta}v_0\|_{L^{q,\infty}(0,T;L^p)}\leq C \|v_0\|_{\BP}.
\end{align}
As $\PP$ is a bounded linear operator on $L^s$ for all $s\in (1,\infty)$, we have
\begin{align*}
    \|u(t)\|_{L^p}\leq \|u(t)-v(t)\|_{L^p}+\|v(t)\|_{L^p}\leq C\|\rho(t)-1\|_{L^{\infty}}\|u(t)\|_{L^p}+\|v(t)\|_{L^p},
\end{align*}
which along with \eqref{assumption.rho(t)}  gives
\begin{align}\label{estimate.u_v control}
  \|u(t)\|_{L^p}\leq C\|v(t)\|_{L^p}, \quad \|u(t)-v(t)\|_{L^p}\leq C\varepsilon_0 \|v(t)\|_{L^p}.
\end{align}
Hence, by \eqref{estimate.nabla Af}, \eqref{eq.v_mild solution}, \eqref{estimate.heat v_0} and \eqref{estimate.u_v control}, we obtain
\begin{align*}
    \|v\|_{L^{q,\infty}(0,T;L^p)}
    &\lesssim \|e^{t\Delta}v_0\|_{L^{q,\infty}(0,T;L^p)}+
    \|u-v\|_{L^{q,\infty}(0,T;L^p)}+\|\rho u\otimes u\|_{L^{\f q2,\infty}(0,T;L^{\f p2})}\\
    &\leq C \|v_0\|_{\BP}+C\varepsilon_0\|v\|_{L^{q,\infty}(0,T;L^p)}+C\|\rho\|_{L^{\infty}(0,T;L^{\infty})}\|v\|^2_{L^{q,\infty}(0,T;L^p)}.
\end{align*}
This means that
\begin{align}\label{v1}
 \|v\|_{L^{q,\infty}(0,T;L^p)}\leq C (\|v_0\|_{\BP}+\|v\|^2_{L^{q,\infty}(0,T;L^p)}).
\end{align}
%As $\PP$ is a Fourier multiplier of degree $0$, it follows that
It follows from \eqref{assumption.rho(t)} that
\begin{align*}
 \|v_0\|_{\BP}=\|m_0\|_{\BP}<\varepsilon_0.
\end{align*}
As a result, we deduce from \eqref{estimate.u_v control}, \eqref{v1} and a continuity argument that
\begin{align*}
 \|u\|_{L^{q,\infty}(0,T;L^p)}\leq C \|v\|_{L^{q,\infty}(0,T;L^p)}\leq C\|m_0\|_{\BP}.
\end{align*}

This completes the proof of Proposition \ref{prop.u_L^q,infty(L^p)}.
\end{proof}

\begin{rmk}
   For any $p_1 \in[p,\infty]$ and $q_1\in (1,\infty)$ satisfying $\f{3}{p_1}+\f{2}{q_1}=1$, by the embedding $\BP(\R^3)\hookrightarrow \dot{B}_{p_1,\infty}^{-1+\f{3}{p_1}}(\R^3)$,  we also have
  \begin{align}\label{estimate.u_L^tilde q,infty(L^tilde p)}
    \|u\|_{L^{q_1,\infty}(0,T;L^{p_1})}\leq C\|m_0\|_{\dot{B}_{p_1,\infty}^{-1+\f{3}{p_1}}}\leq C\|m_0\|_{\BP}.
\end{align}
\end{rmk}

\begin{prop}\label{prop.tu_t L^q,infty(L^p)}
 Under the assumptions of Proposition \ref{prop.u_L^q,infty(L^p)}, we have
 \begin{align}\label{estimate.tu}
    \|tu\|_{L^{\infty}(0,T;\dot{B}_{p,\infty}^{1+\f3p})}+\bigl\|\left((tu)_t, tu_t, t\dot{u}, t\nabla^2u, t\nabla P\right)\bigr\|_{L^{q,\infty}(0,T;L^p)}\leq C\|m_0\|_{\BP}.
 \end{align}
Moreover, there holds
\begin{align}\label{decay.u_L^p}
t^{\f1q}\|u(t)\|_{L^p}+t^{\f12}\|u(t)\|_{L^{\infty}}+t^{\f1q+\f12}\|\nabla u(t)\|_{L^p}\leq C\|m_0\|_{\BP}.
\end{align}
\end{prop}

\begin{proof}
    Multiplying both sides of \eqref{INS} by time $t$ yields
    \begin{align*}
        (tu)_t-\Delta(tu)+\nabla(tP)=(1-\rho)(tu)_t+\rho u-t\rho u\cdot\nabla u.
    \end{align*}
    Then by taking advantage of Lemma \ref{lem.maximal regularity}, we get
    \begin{align*}
    &\|tu\|_{L^{\infty}(0,T;\dot{B}_{p,\infty}^{1+\f3p})}+\|((tu)_t, t\nabla^2u, t\nabla P)\|_{L^{q,\infty}(0,T;L^p)}\\
    & \lesssim \|(1-\rho)(tu)_t\|_{L^{q,\infty}(0,T;L^p)}+\|\rho u\|_{L^{q,\infty}(0,T;L^p)}+\|t\rho u\cdot\nabla u\|_{L^{q,\infty}(0,T;L^p)}\\
    & \lesssim \|\rho-1\|_{L^{\infty}(0,T;L^{\infty})}\|(tu)_t\|_{L^{q,\infty}(0,T;L^p)}+\|\rho\|_{L^{\infty}(0,T;L^{\infty})} \bigl(\|u\|_{L^{q,\infty}(0,T;L^p)}+\|t u\cdot\nabla u\|_{L^{q,\infty}(0,T;L^p)}\bigr).
    \end{align*}
    Thanks to \eqref{assumption.rho(t)}, the first term of the right hand side can be absorbed by the left hand side. And by the interpolation inequality, we have
    \begin{align}\label{estimate.Lp interpolation}
        \|u\|_{L^{\infty}}\lesssim \|u\|^{\f{p}{p+3}}_{L^p}\|u\|^{\f{3}{p+3}}_{\dot{B}_{p,\infty}^{1+\f3p}} \andf \|\nabla u\|_{L^p}\lesssim \|u\|^{\f{3}{p+3}}_{L^p}\|u\|^{\f{p}{p+3}}_{\dot{B}_{p,\infty}^{1+\f3p}},
    \end{align}
    which implies
    \begin{align*}
     \|t u\cdot\nabla u\|_{L^{q,\infty}(0,T;L^p)}\lesssim  \|u\|_{L^{q,\infty}(0,T;L^p)}\|tu\|_{L^{\infty}(0,T;\dot{B}_{p,\infty}^{1+\f3p})}.
    \end{align*}
    So, we get
    \begin{align*}
    \|tu\|_{L^{\infty}(0,T;\dot{B}_{p,\infty}^{1+\f3p})}+\|((tu)_t, t\nabla^2u, t\nabla P)\|_{L^{q,\infty}(0,T;L^p)}\leq C\|u\|_{L^{q,\infty}(0,T;L^p)}\bigl(1+\|tu\|_{L^{\infty}(0,T;\dot{B}_{p,\infty}^{1+\f3p})}\bigr).
    \end{align*}
    By Proposition \ref{prop.u_L^q,infty(L^p)}, we know that $\|u\|_{L^{q,\infty}(0,T;L^p)}$ is small, from which, we infer
    \begin{align*}
    &\|tu\|_{L^{\infty}(0,T;\dot{B}_{p,\infty}^{1+\f3p})}+\|((tu)_t, t\nabla^2u, t\nabla P)\|_{L^{q,\infty}(0,T;L^p)} \leq C\|m_0\|_{\BP}.
    \end{align*}
    As $t{u}_t=(tu)_t-u$, $t\dot{u}=(tu)_t-u+tu\cdot\nabla u$, we  get,  by using \eqref{estimate.Lp interpolation},  that
        \begin{align*}
        \|tu_t, t\dot{u}\|_{L^{q,\infty}(0,T;L^p)}{\lesssim }\|(tu)_t\|_{L^{q,\infty}(0,T;L^p)}+ \|u\|_{L^{q,\infty}(0,T;L^p)}+\|tu\|_{L^{\infty}(0,T;\dot{B}_{p,\infty}^{1+\f3p})}\|u\|_{L^{q,\infty}(0,T;L^p)},
    \end{align*}
which leads to  \eqref{estimate.tu}.

 Notice that
 \begin{align*}
     tu(t)=\int_0^t (su)_s\,\ds,
 \end{align*}
 we get by H\"older inequality of Lorentz space that
 \begin{align*}
     t\|u(t)\|_{L^p}\leq \int_0^t \|(tu)_t\|_{L^p}\,\ds\leq \|1\|_{L^{q',1}(0,t)}\|(tu)_t\|_{\LQQ(0,t;L^p)}\leq q't^{1-\f1q}\|(tu)_t\|_{\LQQ(0,t;L^p)},
 \end{align*}
 where $\f1q+\f{1}{q'}=1$. This implies $t^{\f1q}\|u(t)\|_{L^p}
 \leq C\|m_0\|_{\BP}$.
 By \eqref{estimate.Lp interpolation} and \eqref{estimate.tu}, we get
 \begin{align*}
 \|u\|_{L^{\infty}}\leq Ct^{-\f12}{\|m_0\|_{\BP}} \andf
 \|\nabla u\|_{L^p}\leq Ct^{-\f1q-\f12}{\|m_0\|_{\BP}}.
 \end{align*}
 We thus obtain \eqref{decay.u_L^p}.
 \end{proof}

\begin{prop}\label{prop.t^{1+1/2q}u L^2q,infty(L^p)}
 Under the assumptions of Proposition \ref{prop.u_L^q,infty(L^p)}, we have
 \begin{align}\label{estimate.t^1+1/2q u}
    \|t^{1+\f{1}{2q}}u\|_{L^{\infty}(0,T;\dot{B}_{p,\infty}^{1+\f3p+\f1q})}+\bigl\|\bigl((t^{1+\f{1}{2q}}u)_t, t^{1+\f{1}{2q}}\nabla^2u, t^{1+\f{1}{2q}}\nabla P\bigr)\bigr\|_{L^{2q,\infty}(0,T;L^p)}\leq C\|m_0\|_{\BP}.
 \end{align}
Moreover, there holds
\begin{align}\label{decay.nabla u_L^infty}
t\|\nabla u\|_{L^{\infty}}\leq C\|m_0\|_{\BP}.
\end{align}
\end{prop}

\begin{proof} We first get, by
 multiplying both sides of \eqref{INS} by  $t^{1+\f{1}{2q}},$ that
    \begin{align*}
        (t^{1+\f{1}{2q}}u)_t-\Delta(t^{1+\f{1}{2q}}u)+\nabla(t^{1+\f{1}{2q}}P)=(1-\rho)(t^{1+\f{1}{2q}}u)_t+(1+1/2q)t^{\f{1}{2q}} \rho u-t^{1+\f{1}{2q}}\rho u\cdot\nabla u.
    \end{align*}
    By taking advantage of Lemma \ref{lem.maximal regularity}, we get
    \begin{align*}
    &\|t^{1+\f{1}{2q}}u\|_{L^{\infty}(0,T;\dot{B}_{p,\infty}^{1+\f3p+\f1q})}+\bigl\|\bigl((t^{1+\f{1}{2q}}u)_t, t^{1+\f{1}{2q}}\nabla^2u, t^{1+\f{1}{2q}}\nabla P\bigr)\bigr\|_{L^{2q,\infty}(0,T;L^p)}\\
    & \lesssim \|(1-\rho)(t^{1+\f{1}{2q}}u)_t\|_{L^{2q,\infty}(0,T;L^p)}+\|t^{\f{1}{2q}}\rho u\|_{L^{2q,\infty}(0,T;L^p)}+\|t^{1+\f{1}{2q}}\rho u\cdot\nabla u\|_{L^{2q,\infty}(0,T;L^p)}\\
    & \lesssim \|\rho-1\|_{L^{\infty}(0,T;L^{\infty})}\|(t^{1+\f{1}{2q}}u)_t\|_{L^{2q,\infty}(0,T;L^p)}\\
    &\qquad\qquad\qquad\qquad+\|\rho\|_{L^{\infty}(0,T;L^{\infty})} \bigl(\|t^{\f{1}{2q}}u\|_{L^{2q,\infty}(0,T;L^p)}+\|t^{1+\f{1}{2q}} u\cdot\nabla u\|_{L^{2q,\infty}(0,T;L^p)}\bigr).
    \end{align*}
    Thanks to \eqref{assumption.rho(t)}, the first term of the right hand side can be absorbed by the left hand side. By \eqref{decay.u_L^p}, we have
    \begin{align*}
        \|t^{\f{1}{2q}}u\|_{L^{2q,\infty}(0,T;L^p)}\leq \|t^{-\f{1}{2q}}\|_{L^{2q,\infty}(0,T)}\|t^{\f1q}u(t)\|_{L^{\infty}(0,T;L^p)}\leq C\|m_0\|_{\BP},
    \end{align*}
and
    \begin{align*}
     \|t^{1+\f{1}{2q}} u\cdot\nabla u\|_{L^{2q,\infty}(0,T;L^p)}\leq  \bigl(t^{1+\f{1}{2q}}\|u(t)\|_{L^{\infty}}\|\nabla u\|_{L^p}\bigr)_{L^{2q,\infty}(0,T)}\leq C\|m_0\|^2_{\BP}.
    \end{align*}
    This proves \eqref{estimate.t^1+1/2q u}.
    By the interpolation inequality, \eqref{decay.u_L^p} and \eqref{estimate.t^1+1/2q u}, we obtain
    \begin{align*}
        \|\nabla u(t)\|_{L^{\infty}}\lesssim \|u(t)\|^{\f{1}{q+1}}_{L^{\infty}}\|u(t)\|^{\f{q}{q+1}}_{\dot{B}_{p,\infty}^{1+\f3p+\f1q}}\leq Ct^{-1}\|m_0\|_{\BP}.
    \end{align*}

    This completes the proof of Proposition \ref{prop.t^{1+1/2q}u L^2q,infty(L^p)}.
\end{proof}

In order to derive higher order time weighted estimates, we consider the equation of $\dot{u}$.
 Indeed by taking $D_t{\eqdef}\p_t+u\cdot \nabla$ to \eqref{INS}, we find
\begin{align}\label{eq.dot u}
    \rho \ddot{u}-\Delta\dot{u}+\nabla\dot{P}=f \with f{\eqdef}-\Delta u\cdot\nabla u-2\nabla u\cdot \nabla^2u+\nabla u\cdot\nabla P.
\end{align}

\begin{prop}\label{prop.t^2dotu L^q,infty(L^p)}
 Under the assumptions of Proposition \ref{prop.u_L^q,infty(L^p)}, we have
 \begin{align}\label{estimate.t^2dotu}
    \|t^2\dot{u}\|_{L^{\infty}(0,T;\dot{B}_{p,\infty}^{1+\f3p})}+\|((t^2\dot{u})_t, t^2\nabla^2\dot{u})\|_{L^{q,\infty}(0,T;L^p)}\leq C\|m_0\|_{\BP}.
 \end{align}
Furthermore, we have
\begin{align*}%\label{estimate.t3/2nabla dot u}
\|t^{\f32}\nabla \dot{u}\|_{L^{q,\infty}(0,T;L^p)}\leq C\|m_0\|_{\BP}.
\end{align*}
\end{prop}

\begin{proof}
   It follows from  \eqref{eq.dot u} that $t^2\dot{u}$ verifies
    \begin{align*}
     \rho (t^2\dot{u})_t-\Delta(t^2\dot{u})+\nabla(t^2\dot{P})=-t^2\rho u\cdot\nabla\dot{u}+2t\rho\dot{u}+t^2f.
    \end{align*}
    Recalling that
    \begin{align*}
        \PP{\eqdef}\mathrm{Id}+\nabla(-\Delta)^{-1}\dive \andf \Q{\eqdef}-\nabla(-\Delta)^{-1}\dive,
    \end{align*}
    we observe that
    \begin{align*}
        \nabla(t^2\dot{P})=\Q\Bigl(-t^2\rho u\cdot\nabla\dot{u}+2t\rho\dot{u}+t^2f-\rho (t^2\dot{u})_t+\Delta(t^2\dot{u})\Bigr).
    \end{align*}
   Due to $\dive u=0$, we can write
    \begin{align*}
        \dive \dot{u}=\sum_{1\leq i,j\leq 3}\p_iu^j\p_ju^i=\mathrm{Tr}(\nabla u\cdot \nabla u),
    \end{align*}
    hence we get
    \begin{align*}
        \Q(t^2\Delta\dot{u})=t^2\nabla \mathrm{Tr}(\nabla u\cdot \nabla u),
    \end{align*}
    and we have
    \begin{align*}
     \Q\bigl(\rho (t^2\dot{u})_t\bigr)=\Q\Bigl((\rho-1) (t^2\dot{u})_t+2t\dot{u}+t^2u\cdot\nabla u_t+t^2u_t\cdot\nabla u\Bigr),\quad \Q(u_t\cdot\nabla u)=\Q(u\cdot\nabla u_t).
    \end{align*}
 Finally, we  obtain
    \begin{align}
     (t^2\dot{u})_t-\Delta(t^2\dot{u})=\PP[(1-\rho)(t^2\dot{u})_t&-t^2\rho u\cdot\nabla\dot{u}+2t\rho\dot{u}+t^2f]\label{eq.t^2dotu}\\
     &+\Q(2t\dot{u}+{2t^2u_t\cdot\nabla u})-t^2\nabla\mathrm{Tr}(\nabla u\cdot \nabla u).\nonumber
    \end{align}

At this point, we use the maximal regularity estimate for the heat equation stated in [Proposition 2.1, \cite{DMT}], and the boundedness of $\PP$ and $\Q$ on $L^p$ to conclude that
\begin{align*}
    &\|t^2\dot{u}\|_{L^{\infty}(0,T;\dot{B}_{p,\infty}^{1+\f3p})}+\|((t^2\dot{u})_t, t^2\nabla^2\dot{u})\|_{L^{q,\infty}(0,T;L^p)}\\
    &\quad\leq {C\|(\rho-1)(t^2\dot{u})_t\|_{L^{q,\infty}(0,T;L^p)}+C}\|t^2\rho u\cdot\nabla\dot{u}{-2t\rho\dot{u}-t^2f}\|_{L^{q,\infty}(0,T;L^p)}\\
    &\qquad+C\|t\dot{u}+t^2u_t\cdot\nabla u\|_{L^{q,\infty}(0,T;L^p)}+C\|t^2\nabla\mathrm{Tr}(\nabla u\cdot \nabla u)\|_{L^{q,\infty}(0,T;L^p)}.
\end{align*}
Thanks to \eqref{assumption.rho(t)}, the first term of the right hand side can be absorbed by the left hand side. By \eqref{decay.u_L^p} and the interpolation inequality, we get
\begin{align*}
 C\|t^2\rho u\cdot\nabla\dot{u}\|_{L^{q,\infty}(0,T;L^p)}&\leq C\|\rho\|_{L^{\infty}(0,T;L^{\infty})} \bigl(t^2\|u\|_{L^{\infty}}\|\nabla\dot{u}\|_{L^p}\bigr)_{L^{q,\infty}(0,T)}\\
 &\leq C \|m_0\|_{\BP}\bigl(t^{\f32}\|\dot{u}\|^{\f12}_{L^{p}}\|\nabla^2\dot{u}\|^{\f12}_{L^p}\bigr)_{L^{q,\infty}(0,T)}\\
 &\leq C \|m_0\|_{\BP}\|t\dot{u}\|^{\f12}_{L^{q,\infty}(0,T;L^p)}\|t^2\nabla^2\dot{u}\|^{\f12}_{L^{q,\infty}(0,T;L^p)}\\
 &\leq \f14\|t^2\nabla^2\dot{u}\|_{L^{q,\infty}(0,T;L^p)}+C\|m_0\|^2_{\BP}\|t\dot{u}\|_{L^{q,\infty}(0,T;L^p)}.
\end{align*}
By {the definition of $f$ in \eqref{eq.dot u}},  we have
\begin{align*}
   &\|t^2f\|_{L^{q,\infty}(0,T;L^p)}+\|t^2\nabla\mathrm{Tr}(\nabla u\cdot \nabla u)\|_{L^{q,\infty}(0,T;L^p)}+\|t^2u_t\cdot\nabla u\|_{L^{q,\infty}(0,T;L^p)}\\
   &\leq C\|t\nabla u\|_{L^{\infty}(0,T;L^{\infty})}\|(t\nabla^2u,t\nabla P,tu_t)\|_{L^{q,\infty}(0,T;L^p)}.
\end{align*}
Then \eqref{estimate.t^2dotu} follows {from \eqref{estimate.tu} and \eqref{decay.nabla u_L^infty}}.
Finally, by the interpolation inequality, we obtain
\begin{align*}
    \|t^{\f32}\nabla\dot{u}\|_{L^{q,\infty}(0,T;L^p)}\leq C\|t\dot{u}\|^{\f12}_{L^{q,\infty}(0,T;L^p)}\|t^2\nabla^2\dot{u}\|^{\f12}_{L^{q,\infty}(0,T;L^p)}\leq C\|m_0\|_{\BP}.
\end{align*}

This completes the proof of Proposition \ref{prop.t^2dotu L^q,infty(L^p)}.
\end{proof}

\subsection{Existence part of Theorem \ref{thm.INS_Cannone-Meyer-Planchon}}\label{Subsec.2.2}
This subsection is devoted to the proof of existence part  in Theorem \ref{thm.INS_Cannone-Meyer-Planchon}.

For $\varepsilon\in(0,1)$, we consider
$$ m_0^{\varepsilon}\in C^{\infty}(\R^3)
\andf \rho_0^{\varepsilon}\in C^{\infty}(\R^3) \with \|\rho_0^{\varepsilon}-1\|_{L^{\infty}}\leq \|\rho_0-1\|_{L^\infty} $$
such that
\begin{equation*}
m_0^{\varepsilon}\to m_0 \text{~in~} \BP(\R^3), \quad \rho_0^{\varepsilon}\rightharpoonup \rho_0 \text{~in~} L^{\infty} \text{~weak-*},
\andf~ \rho_0^{\varepsilon}\to \rho_0 \text{~in~} L^{p}_{\text{loc}}(\R^3) \text{~if~} p<\infty,
\end{equation*}
as $\varepsilon\to 0+$. In light of the classical strong solution theory for the system \eqref{INS}, there exists a unique global smooth solution $(\rho^{\varepsilon}, u^{\varepsilon}, P^{\varepsilon})$
corresponding to data $(\rho_0^{\varepsilon}, m_0^{\varepsilon})$. Thus, the triple $(\rho^{\varepsilon}, u^{\varepsilon}, P^{\varepsilon})$  satisfies all the a priori estimates of Subsection \ref{Subsec.2.1} uniformly with respect to $\varepsilon\in(0,1)$. Hence, $(\rho^{\varepsilon}, u^{\varepsilon})$ converges weakly (or weakly-*) to a limit $(\rho, u)$ as $\varepsilon\to0+$, up to subsequence. To show that the limit solves \eqref{INS} weakly, in view of standard compactness arguments, it suffices to prove that
\begin{equation}\label{Eq.convergence}
    \lim_{\varepsilon\to0+}\int_0^t\langle \rho^\varepsilon u^\varepsilon\otimes u^\varepsilon(s)-\rho u\otimes u(s), \nabla\varphi\rangle\,\ds=0\quad\forall\ t>0
\end{equation}
for any divergence-free function $\varphi\in C_c^\infty([0, +\infty)\times\R^3)$.  Since
\begin{align*}
    \int_0^t|\langle \rho^\varepsilon u^\varepsilon\otimes u^\varepsilon(s),\nabla\varphi\rangle|\,\ds
    &\leq\|\rho_0\|_{L^{\infty}} \|u^\varepsilon\|^2_{L^{q,\infty}(0,t;L^p)}\|\nabla \varphi\|_{L^{\f{q}{q-2},1}(0,t;L^{\f{p}{p-2}})}\leq Ct^{\f{q-2}{q}},\\
    \int_0^t|\langle \rho u\otimes u(s), \nabla\varphi\rangle|\,\ds&\leq \|\rho_0\|_{L^{\infty}} \|u\|^2_{L^{q,\infty}(0,t;L^p)}\|\nabla \varphi\|_{L^{\f{q}{q-2},1}(0,t;L^{\f{p}{p-2}})}\leq Ct^{\f{q-2}{q}},
\end{align*}
for all $\varepsilon\in(0,1)$, $t>0$, where $C>0$ is a constant {depending} on $\|\rho_0\|_{L^{\infty}},\|m_0\|_{\BP}$.
{Now we fix $t>0$ and let $\eta>0$.} Taking $\delta_0{\eqdef}\min\{t/2, (\eta/(4C))^{\f{q}{q-2}}\}>0$, we find
\[\int_0^{\delta_0}\left|\langle \rho^\varepsilon u^\varepsilon\otimes u^\varepsilon(s)-\rho u\otimes u(s), \nabla\varphi\rangle\right|\,\ds<\frac\eta2,\quad\forall\ \varepsilon\in(0,1).\]
To prove \eqref{Eq.convergence}, it suffices to show that
\begin{equation}\label{Eq.convergence_delta}
    \lim_{\varepsilon\to0+}\int_{\delta_0}^t\langle \rho^\varepsilon u^\varepsilon\otimes u^\varepsilon(s)-\rho u\otimes u(s), \nabla\varphi\rangle\,\ds=0.
\end{equation}
Indeed, assuming that $\operatorname{supp}_x\varphi(s)\subset B(0, R)$ for all $s\in[\delta_0, t]$, on one hand, by the weak-* convergence of $\rho^\varepsilon $ to $\rho $ in $L^{\infty}([0,T]\times B_R)$ (up to subsequence),
we have
\begin{equation}\label{Eq.convergence_delta1}
    \lim_{\varepsilon\to0+}\int_{\delta_0}^t\langle (\rho^\varepsilon-\rho) u\otimes u(s), \nabla\varphi\rangle\,\ds=0.
\end{equation}
{On the other hand, thanks to \eqref{estimate.tu}, we infer that $\{t\p_tu^{\varepsilon}\}$ is uniformly bounded in $L^{q,\infty}(\R^+;L^p)$, from which we know that $\{\p_tu^{\varepsilon}\}$ is uniformly bounded in $L^{2}(\delta_0,t;L^p)$.
By %\eqref{estimate.u_L^q,infty(L^p)}
 \eqref{estimate.tu} and \eqref{decay.u_L^p},
we get that $u^{\varepsilon}$ is uniformly bounded in $L^{\infty}(\delta_0,t;\dot{B}^{1+\f3p}_{p,\infty}\cap L^p)$.} Then by the Ascoli-Arzela theorem, we conclude that
\begin{align*}
    &u^{\varepsilon}\to u \text{~in~} L^{q}_{loc}(\delta_0,t;L^p_{loc}),\\
     &\nabla u^{\varepsilon}\to \nabla u \text{~in~} L^{q}_{loc}(\delta_0,t;L^p_{loc}),
\end{align*}
 which along with the uniform boundedness of $\rho^\varepsilon $ in $L^{\infty}$ implies that
\begin{equation}\label{Eq.convergence_delta2}
    \lim_{\varepsilon\to0+}\int_{\delta_0}^t\left|\langle \rho^\varepsilon (u^\varepsilon\otimes u^\varepsilon-u\otimes u)(s), \nabla\varphi\rangle\right|\,\ds=0.
\end{equation}
Now \eqref{Eq.convergence_delta} follows from \eqref{Eq.convergence_delta1} and \eqref{Eq.convergence_delta2}.

\subsection{Uniqueness part of Theorem \ref{thm.INS_Cannone-Meyer-Planchon}}\label{Subsec.2.3}
The goal of this subsection is to give the proof of uniqueness part in Theorem \ref{thm.INS_Cannone-Meyer-Planchon}.

\begin{proof}[Proof of uniqueness of Theorem \ref{thm.INS_Cannone-Meyer-Planchon}]
Let $p\in (3,6),~ p'\in (1,\infty)$ satisfy $\f1p+\f{1}{p'}=1$.
 Let $(\rho, u, \nabla P)$ and $(\bar \rho, \bar u, \nabla\bar P)$ be two solutions to \eqref{INS} with the same initial data $(\rho_0, m_0)$ and satisfy all the hypotheses listed in Theorem \ref{thm.INS_Cannone-Meyer-Planchon}. We denote $\delta\!\rho{\eqdef}\rho-\bar{\rho}$ and $\delta\!u{\eqdef}u-\bar{u}$. Then $(\delta\!\rho, \delta\!u)$ satisfies
    \begin{equation}\label{Eq.delta u}
     \left\{
     \begin{array}{l}
     \partial_t\delta\!\rho+\bar{u}\cdot\nabla \delta\!\rho+\delta\!u\cdot\nabla\rho=0,\qquad (t,x)\in \mathbb{R}^{+}\times\mathbb{R}^{3},\\
     \rho(\partial_t\delta\!u+u\cdot\nabla \delta\!u)-\Delta \delta\!u+\nabla \delta\!P=-\delta\!\rho\dot{\bar{u}}-\rho\delta\!u\cdot\nabla\bar{u},\\
     \dive \delta\!u = 0,\\
     (\delta\!\rho, \rho\delta\!u)|_{t=0} =(0, 0).
     \end{array}
     \right.
     \end{equation}
Here $\dot{\bar{u}}{\eqdef}\pa_t\bar{u}+\bar{u}\cdot\nabla\bar{u}$.
Let $\delta\!v{\eqdef}\PP(\rho\delta\!u)$, then $\delta\!v$ satisfies
\begin{align}\label{eq.delta v}
    \p_t\delta\!v-\Delta \delta\!v=\Delta(\delta\!u-\delta\!v)-\PP\dive (\rho u\otimes \delta\!u)-\PP(\delta\!\rho\dot{\bar{u}}+\rho\delta\!u\cdot\nabla\bar{u}).
\end{align}
And similar to \eqref{estimate.u_v control}, we have
\begin{align}\label{estimate.delta u_v control}
 \|\delta\!u(t)\|_{L^p}\leq C\|\delta\!v(t)\|_{L^p} \andf \|\delta\!u(t)-\delta\!v(t)\|_{L^p}\leq C\varepsilon_0\|\delta\!v(t)\|_{L^p}.
\end{align}
Let $T>0$. For all $t\in(0,T]$, set
\begin{align}
&X(t){\eqdef}\sup_{s\in (0,t]}\|\delta\!\rho(s)\|_{\dot{W}^{-1,p}},\label{X_t}\\
&Y(t){\eqdef}\| \delta\!v\|_{L^{q,\infty}(0,t;L^p)}.\label{Y_t}
\end{align}
For any $\varphi\in C^{\infty}_c(\R^3)$, as $\|\nabla \bar u(t)\|_{L^{\infty}} \leq C\varepsilon_0 t^{-1}$, by the classical theory on transport equations (\cite{BCD}, Theorem 3.2), there exists a unique $\phi(s,x)$ solving
\begin{align*}
\partial_s\phi+\bar{u}\cdot\nabla\phi=0,\quad (s,x)\in(0,t]\times\R^3, \qquad \phi(t,x)=\varphi(x),
\end{align*}
with the estimate
\begin{align}\label{transport}
\|\nabla\phi(s)\|_{L^{p'}}\leq C\|\nabla\varphi\|_{L^{p'}}\exp\bigl(\int_s^t\|\nabla\bar{u}(\tau)\|_{L^{\infty}}\,\mathrm{d}\tau\bigr),\quad\forall\ s\in(0, t].
\end{align}
Note that(here $\f1q+\f{1}{q'}=1$)
\begin{align}\label{eq.transport_integrable}
\exp\bigl(\int_s^t\|\nabla\bar{u}(\tau)\|_{L^{\infty}}\,\mathrm{d}\tau\bigr)\leq C(t/s)^{C\varepsilon_0} \andf \|(t/s)^{C\varepsilon_0}\|_{L^{q',1}(0,t)}\leq Ct^{1-\f1q}.
\end{align}
 Using the first equation of \eqref{Eq.delta u} and integration by parts, we have
    \begin{align*}
\frac{\mathrm{d}}{\ds} \langle\delta\!\rho(s), \phi(s)\rangle
&=\langle -\dive\{ \delta\!\rho\bar{u}\},\phi\rangle-\langle \dive \{\rho \delta\!u\},\phi\rangle-\langle \delta\!\rho,\bar{u}\cdot\nabla\phi\rangle\nonumber\\
&=\langle \rho\delta\!u,\nabla\phi(s)\rangle,\quad\forall\ s\in(0, t],
\end{align*}
then integrating over $(0, t)$, we get by \eqref{transport} and \eqref{eq.transport_integrable} that
    \begin{align}
      |\langle\delta\!\rho(t), \varphi\rangle|
      &\leq \int_0^t|\langle\rho\delta\!u(s),\nabla\phi(s)\rangle|\,\ds\leq\int_0^t\|\rho\delta\!u(s)\|_{L^p}\|\nabla\phi(s)\|_{L^{p'}}\,\ds\label{delta rho.phi}\\
     &\leq C t^{1-\f1q}\|\rho_0\|_{L^{\infty}} \|\delta\!u\|_{L^{q,\infty}(0,t;L^p)}\|\nabla\varphi\|_{L^{p'}}\nonumber
    \end{align}
    for all $t\in(0, T]$.
Then it follows from \eqref{X_t} and \eqref{Y_t} that
\begin{align}\label{Eq.X<Y}
  X(t)\leq Ct^{1-\f1q}Y(t),\quad\forall\ t\in(0, T].
\end{align}

As for \eqref{eq.delta v}, we can write
\begin{align*}
\delta\!v(t,x)=\int_0^t e^{(t-s)\Delta}\Bigl(\Delta(\delta\!u-\delta\!v)-\PP\dive(\rho u\otimes \delta\!u)-\PP(\delta\!\rho\dot{\bar{u}}+\rho\delta\!u\cdot\nabla\bar{u})\Bigr)(s,x)\,\ds.
\end{align*}
Let us set $F{\eqdef}-(-\D)^{-1}\delta\!\rho$, which implies
\begin{align}\label{def.F}
    \delta\!\rho\dot{\bar{u}}=\D F\cdot \dot{\bar{u}}=\dive (\nabla F\cdot \dot{\bar{u}})-\nabla F\cdot \nabla\dot{\bar{u}} \andf \|\nabla F(s)\|_{L^p}\sim \|\delta\!\rho(s)\|_{\dot{W}^{-1,p}}.
\end{align}
Let $\f1r=\f{2}{q}+\f12$. Then by \eqref{estimate.delta u_v control} and Corollary \ref{cor.convolution}, we can obtain
\begin{align}
 \|\delta\!v\|_{L^{q,\infty}(0,T;L^p)}&\leq C \|\delta\!u-\delta\!v\|_{L^{q,\infty}(0,T;L^p)}+\|\rho u\otimes \delta\!u\|_{L^{q/2,\infty}(0,T;L^{p/2})}+\|\nabla F\cdot\dot{\bar{u}}\|_{L^{q/2,\infty}(0,T;L^{p/2})}\nonumber\\
 &\quad+\|\nabla F\cdot\nabla \dot{\bar{u}}\|_{L^{r,\infty}(0,T;L^{p/2})}+\|\rho\delta\!u\cdot\nabla \bar{u}\|_{L^{r,\infty}(0,T;L^{p/2})}.\label{uni1.delta v}
\end{align}
 We know from  \eqref{estimate.delta u_v control} that the first term in right hand can be absorbed by the left hand. And we get by  H\"older inequality that
\begin{align}\label{estimate.Y1}
 \|\rho u\otimes \delta\!u\|_{L^{q/2,\infty}(0,T;L^{p/2})}\leq \|\rho_0\|_{L^{\infty}}\|u\|_{L^{q,\infty}(0,T;L^p)} \|\delta\!u\|_{L^{q,\infty}(0,T;L^p)}\leq C\varepsilon_0 \|\delta\!v\|_{L^{q,\infty}(0,T;L^p)}.
 \end{align}
Then by \eqref{Eq.X<Y}, \eqref{def.F} and Proposition \ref{prop.tu_t L^q,infty(L^p)}, we have
 \begin{equation}\begin{aligned}\label{estimate.Y2}
 \|\nabla F\cdot\dot{\bar{u}}\|_{L^{q/2,\infty}(0,T;L^{p/2})}&\leq \bigl(\|\nabla F(t)\|_{L^p}\|\dot{\bar{u}}(t)\|_{L^p}\bigr)_{L^{q/2,\infty}(0,T)}\leq C\bigl(X(t)\|\dot{\bar{u}}(t)\|_{L^p}\bigr)_{L^{q/2,\infty}(0,T)}\\
 &\leq C Y(T)\|t^{-\f1q}\|_{L^{q,\infty}(0,T)}\|t\dot{\bar{u}}\|_{L^{q,\infty}(0,T;L^p)}\leq C\varepsilon_0Y(T).
 \end{aligned}\end{equation}
 Similarly, let $\f1m=\f1q+\f12$, then by \eqref{Eq.X<Y}, \eqref{def.F} and Proposition \ref{prop.t^2dotu L^q,infty(L^p)}, we have
 \begin{equation}\begin{aligned}\label{estimate.Y3}
 &\|\nabla F\cdot\nabla\dot{\bar{u}}\|_{L^{r,\infty}(0,T;L^{p/2})}\leq \bigl(\|\nabla F(t)\|_{L^p}\|\nabla\dot{\bar{u}}(t)\|_{L^p}\bigr)_{L^{r,\infty}(0,T)}\\
 &\qquad\leq C Y(T)\|t^{-\f1q-\f12}\|_{L^{m,\infty}(0,T)}\|t^{3/2}\nabla\dot{\bar{u}}\|_{L^{q,\infty}(0,T;L^p)}\leq C\varepsilon_0Y(T).
 \end{aligned}\end{equation}
 And by H\"older inequality, {\eqref{decay.u_L^p} and \eqref{estimate.delta u_v control}}, we can obtain
  \begin{equation}\begin{aligned}\label{estimate.Y4}
 \|\rho\delta\!u\cdot\nabla \bar{u}\|_{L^{r,\infty}(0,T;L^{p/2})}&\leq \|\rho_0\|_{L^{\infty}}\|\delta\!u\|_{L^{q,\infty}(0,T;L^p)}\|\nabla\bar{u}\|_{L^{m,\infty}(0,T;L^p)}\\
 &\leq C\varepsilon_0Y(T)\|t^{-\f1q-\f12}\|_{L^{m,\infty}(0,T)}
 \leq C\varepsilon_0Y(T).
 \end{aligned}\end{equation}

 Inserting \eqref{estimate.Y1}, \eqref{estimate.Y2}, \eqref{estimate.Y3} and \eqref{estimate.Y4} into \eqref{uni1.delta v} leads to $Y(T)=0$. Then by \eqref{Eq.X<Y} we know that $X(T)=0$, $\delta\!\rho\equiv 0$ on $[0,T]$. And by \eqref{estimate.delta u_v control}, we finally get $\delta\!u\equiv 0$ on $[0,T]$. Since $T>0$ is arbitrary, we complete the proof of uniqueness on the whole time interval $[0,+\infty)$.
\end{proof}

\section{Weak-Strong uniqueness}\label{Sec.weak-strong}
This section is devoted to proving the weak-strong uniqueness of solutions to \eqref{INS}, as outlined in Theorem \ref{thm.weak-strong uniqueness}.

\subsection{Decomposition of homogeneous Besov spaces}\label{Subsec.3.1}
Here we first recall the decomposition of Besov spaces. This will play a crucial role in the proof of Theorem \ref{thm.weak-strong uniqueness}. 
\begin{lem}[Proposition 1.5, \cite{Barker2017}]\label{lem.Besov decomposition}
Let $p\in (3,\infty)$. There exist $p_2\in (p,\infty)$, $\delta_2\in (0,-s_{p_2})$, $\g_1,\g_2>0$ and $C>0$, each depending only on $p$, such that for each divergence-free vector field $g\in \dot{B}^{s_p}_{p,\infty}(\R^3)$ (here $s_p:=-1+3/p$) and $N>0$, there exist divergence-free fields $g_1^N\in \dot{B}^{s_{p_2}+\delta_2}_{p_2,p_2}(\R^3)\cap \dot{B}^{s_p}_{p,\infty}(\R^3)$ and $g_2^N\in L^2(\R^3)\cap \dot{B}^{s_p}_{p,\infty}(\R^3)$ with the following properties
\begin{align*}
   g=g_1^N+g_2^N,\quad
   \|g_2^N\|_{L^2}\leq CN^{-\g_2}\|g\|_{\dot{B}^{s_p}_{p,\infty}},\quad
   \|g_1^N\|_{\dot{B}^{s_{p_2}+\delta_2}_{p_2,p_2}}\leq CN^{\g_1}\|g\|_{\dot{B}^{s_p}_{p,\infty}}.
\end{align*}
Furthermore,
\begin{align*}
\|g_1^N\|_{\dot{B}^{s_p}_{p,\infty}}+\|g_2^N\|_{\dot{B}^{s_p}_{p,\infty}}\leq C\|g\|_{\dot{B}^{s_p}_{p,\infty}}.
\end{align*}
\end{lem}

Similar to the proof of Lemma \ref{lem.Besov decomposition}, we can give the following important result. Since we require the specific form of this decomposition in the following context, for the completeness, we present the detailed proof here.
\begin{cor}\label{cor.decomposition}
For\footnote{Indeed, here we can let $c<1$ and close to $1$.} $c\in (0,1),~p\in(3,\infty),~\tilde{p}\in (p,\infty)$, we designate
\begin{align*}
\sigma\eqdef (1-c)(1-\f{p}{\tp}),\quad s\eqdef \f12-\f12 (1-c)(p-2).
\end{align*}
Assume that
\begin{align*}
-1+\f{3}{\tp}+\sigma<-1+\f{3}{p}<0<s<\f12.
\end{align*}
 Then for any divergence-free vector field $f\in \dot{B}^{-1+\f3p}_{p,\infty}(\R^3)$, there exist divergence-free fields $f_{1}\in \dot{B}^{-1+\f{3}{\tp}+\sigma}_{\tp,\infty}(\R^3)\cap\dot{B}^{-1+\f3p}_{p,\infty}(\R^3)$, $f_{2}\in \dot{B}^{s}_{2,\infty}(\R^3)\cap \dot{B}^{-1+\f3p}_{p,\infty}(\R^3)$ such that
 \begin{align*}
     &f=f_1+f_2 \with
     \|f_{1}\|^{\tilde{p}}_{\dot{B}^{-1+\f{3}{\tp}+\sigma}_{\tp,\infty}}\leq C\|f\|^{p}_{\dot{B}^{-1+\f{3}{p}}_{p,\infty}},\quad
     \|f_{2}\|^{2}_{\dot{B}^{s}_{2,\infty}}\leq C\|f\|^{p}_{\dot{B}^{-1+\f{3}{p}}_{p,\infty}}.
 \end{align*}
Furthermore,
\begin{align*}
\|f_1\|_{\dot{B}^{-1+\f{3}{p}}_{p,\infty}}+\|f_2\|_{\dot{B}^{-1+\f{3}{p}}_{p,\infty}}\leq C\|f\|_{\dot{B}^{-1+\f{3}{p}}_{p,\infty}}.
\end{align*}
\end{cor}

\begin{proof}
For simplicity, we set $s_p=-1+\f3p,~s_{\tp}=-1+\f{3}{\tp}$. Then we have the following relations
\begin{align}\label{sp}
(s_{\tp}+\s)\tp-ps_p+c(\tp-p)=0,\quad 2s-ps_p+c(2-p)=0.
\end{align}
Let $\DD_j$ be dyadic operator defined by \eqref{defparaproduct},    we write
    \begin{align*}
        \DD_j f=\DD_j f\chi_{|\DD_j f|\leq 2^{jc}}+\DD_j f\chi_{|\DD_j f|>2^{jc}}{\eqdef}\tilde{f}_{j,1}+\tilde{f}_{j,2}.
    \end{align*}
It is  easy to observe that
    \begin{align*}
       & \|\tilde{f}_{j,1}\|^{\tp}_{L^{\tp}}\leq \|\tilde{f}_{j,1}\|^{p}_{L^{p}} \|\tilde{f}_{j,1}\|^{\tp-p}_{L^{\infty}}\leq 2^{jc(\tp-p)}\|\tilde{f}_{j,1}\|^{p}_{L^{p}},\\
       & \|\tilde{f}_{j,2}\|^{2}_{L^{2}}\leq \|\tilde{f}_{j,2}\|^{p}_{L^{p}} \|\tilde{f}^{-1}_{j,2}\chi_{\tilde{f}_{j,2}\neq0}\|^{p-2}_{L^{\infty}}\leq 2^{jc(2-p)}\|\tilde{f}_{j,2}\|^{p}_{L^{p}}.
    \end{align*}
  As $f\in \dot{B}^{s_p}_{p,\infty}$, for $i\in \{1,2\}$ we have
    \begin{align*}
        \|\tilde{f}_{j,i}\|_{L^{p}}\leq 2^{-js_p}\|f\|_{\dot{B}^{s_p}_{p,\infty}} ,
    \end{align*}
    which implies that
    \begin{align*}
      \|\tilde{f}_{j,1}\|^{\tp}_{L^{\tp}}\leq 2^{jc(\tp-p)}2^{-jps_p}\|f\|^{p}_{\dot{B}^{s_p}_{p,\infty}},\qquad
        \|\tilde{f}_{j,2}\|^{2}_{L^{2}}\leq 2^{jc(2-p)}2^{-jps_p}\|f\|^{p}_{\dot{B}^{s_p}_{p,\infty}}.
    \end{align*}

    Next it is well known that for any $f\in \dot{B}^{s_p}_{p,\infty}(\R^3)$,  $\sum_{j=-m}^m\DD_j f$ converges to $f$ in the sense of tempered distributions. Furthermore, using the fact $\DD_{j'}\DD_j f=0$ for $|j-j'|>1$, we can write
    \begin{align*}
    \DD_j f=\widetilde{\Delta}_j\DD_j f=\widetilde{\Delta}_j \tilde{f}_{j,1}+\widetilde{\Delta}_j \tilde{f}_{j,2}\with \widetilde{\Delta}_j{\eqdef}\DD_{j-1}+\DD_j+\DD_{j+1}.
    \end{align*}
    Then for $i\in \{1,2\}$, we set
    \begin{align*}
    f_{j,i}{\eqdef}\PP \widetilde{\Delta}_j\tilde{f}_{j,i}, \quad
     f_1{\eqdef} \sum_{j\in\Z} f_{j,1},\quad f_2{\eqdef}\sum_{j\in\Z}  f_{j,2}.
    \end{align*}
   Note that the Leray projector is a continuous linear operator on the homogeneous Besov spaces, then  we deduce from \eqref{sp} that
    \begin{align*}
        &\|f_1\|^{\tp}_{\dot{B}^{s_{\tp}+\sigma}_{\tp,\infty}}= \sup_{j'\in \Z}2^{j'(s_{\tp}+\s)\tp}\|\DD_{j'}f_1\|^{\tp}_{L^{\tp}}\leq C\sup_{j\in \Z}2^{j(s_{\tp}+\s)\tp}\|f_{j,1}\|^{\tp}_{L^{\tp}}
        \leq C\|f\|^{p}_{\dot{B}^{s_p}_{p,\infty}},\\
        &\|f_2\|^{2}_{\dot{B}^{s}_{2,\infty}}= \sup_{j'\in \Z}2^{2j's}\|\DD_{j'}f_2\|^{2}_{L^{2}}\leq C\sup_{j\in \Z}2^{2js}\|f_{j,2}\|^{2}_{L^{2}}
        \leq C\|f\|^{p}_{\dot{B}^{s_p}_{p,\infty}}.
    \end{align*}

    This completes the  proof of Corollary \ref{cor.decomposition}.      \end{proof}

\subsection{Subcritical estimates}\label{Subsec.3.2}

As $ m_0\in \BP$,  by Corollary \ref{cor.decomposition}, we decompose it as
\begin{align*}
 m_0=({m_0})_{1}&+({m_0})_{2} \with ({m_0})_{1}\in \BPS\cap \dot{B}^{-1+\f3p}_{p,\infty}, \quad ({m_0})_{2}\in \dot{B}^s_{2,\infty}\cap \dot{B}^{-1+\f3p}_{p,\infty},\\
\text{~where~}\quad ({m_0})_{1}&=\dis\sum_{j\in \Z} \PP \widetilde{\D}_j\Bigl((\DD_j m_0)\chi_{|\DD_j  m_0|\leq2^{jc}}\Bigr){\eqdef}\dis\sum_{j\in \Z}({m_0})_{j,1},\\
({m_0})_{2}&=\dis\sum_{j\in \Z} \PP \widetilde{\D}_j\Bigl((\DD_j m_0)\chi_{|\DD_j m_0|>2^{jc}}\Bigr){\eqdef}\dis\sum_{j\in \Z}({m_0})_{j,2}\nonumber.
\end{align*}
Furthermore, we have the following estimates
\begin{align}
    &\|(m_0)_1\|^{\tp}_{\BPS}\leq C\|m_0\|^p_{\BP},\quad \|(m_0)_2\|^{2}_{\dot{B}^s_{2,\infty}}\leq C\|m_0\|^p_{\BP},\nonumber\\
    &\|(m_0)_1\|_{\BP}+\|(m_0)_2\|_{\BP}\leq C\|m_0\|_{\BP}.\label{m_0,1_m_0,2}
\end{align}
Let us denote
\begin{align}\label{def.rho_0u_0N}
m_{0,N}{\eqdef} m_0-\sum_{j=N}^{\infty} ({m_0})_{j,2}.
\end{align}
 %As $({m_0})_2\in \dot{B}^s_{2,\infty}$, 
 Then we have
\begin{align}
    \|m_{0,N}\|_{\BPS}&=\|({m_0})_{1}+\sum_{j=-\infty}^{N-1} ({m_0})_{j,2}\|_{\BPS}\nonumber\\
    &\leq \|(m_0)_1\|_{\BPS}+C\sup_{j\leq N+2}2^{j(-1+\f{3}{\tp}+\sigma)}\|\DD_{j}m_0\|_{L^{\tp}}\label{estimate2.rho_0u_0N}\\
    &\leq C\|m_0\|^{\f{p}{\tp}}_{\BP}+C2^{N\sigma}\|m_0\|_{\BP}\leq C\varepsilon_0^{\f{p}{\tp}}2^{N\sigma}\nonumber.
\end{align}
Based on the above decomposition, we can provide some subcritical estimates for \eqref{INS}.
\begin{prop}\label{prop2.u}
Let $ p\in (3,\infty),~\tp\in (p,\infty),~q\in(1,\infty),~\tq\in (1,\infty),~\bq\in(\tq, \infty)$ and $\sigma\in(0,1)$ satisfy
\begin{align*}
\f3{p}+\f2{q}=1,\quad\f3{\tp}+\f2{\tq}=1, \quad \f{2}{\bq}+\s=\f{2}{\tq}.
\end{align*}
There exists a constant $\varepsilon_0>0$ such that if
\begin{align}\label{R_assumption.rho(t)}
    \|\varrho_0-1\|_{L^{\infty}}+\|n_0\|_{\BP(\R^3)}<\varepsilon_0,
\end{align}
then the following system
	\begin{equation}\label{Eq.linear_NS}
		\begin{cases}
            \pa_t\varrho+R\cdot\nabla\varrho=0,\\
			\varrho(\pa_tR+R\cdot\nabla R)-\Delta R+\nabla P=0,\\
			\dive R=0,\\
			\varrho|_{t=0}=\varrho_0, ~\PP(\varrho R)|_{t=0}=n_0\in\BP\cap \BPS(\R^3)
		\end{cases}
	\end{equation}
	admits a solution, $(\varrho,R),$
on $\R^+\times \R^3$ with the following estimate
\begin{align}\label{estimate2.u}
    \|R\|_{L^{\bq,\infty}(\R^+;L^{\tp})}+\|t^{-\f{\s}{2}}R\|_{L^{\tq,\infty}(\R^+;L^{\tp})}\leq C\|n_0\|_{\BPS},
\end{align}
where the constant $C>0$ depends only on $\|\varrho_0\|_{L^{\infty}}$.
\end{prop}

\begin{proof}
Firstly, Theorem \ref{thm.INS_Cannone-Meyer-Planchon} ensures the global existence of the solution for $(\varrho,R)$, which satisfies all the a priori estimates in Subsection \ref{Subsec.2.1}.

Let $W{\eqdef}\PP(\varrho R)$ with $W_0= n_0$, then $W$ satisfies
\begin{equation*}%\label{eq.W}
     \left\{
     \begin{array}{l}
     \partial_tW-\Delta W=\Delta(R-W)-\PP\dive(\varrho R\otimes R), \quad (t,x)\in \R^{+}\times\R^3\\
     \dive W=0,\\
     W|_{t=0}=W_0,
     \end{array}
     \right.
\end{equation*}
then we can write
\begin{align}\label{eq.W_mild solution}
    W(t,x)=e^{t\Delta}W_0+\int_0^t e^{(t-s)\Delta}\Bigl(\Delta(R-W)-\PP\dive(\varrho R\otimes R)\Bigr)(s,x)\,\ds.
\end{align}
Since $-1+\f{3}{\tp}+\s=-\f{2}{\bq}$, according to the equivalent definition of Besov spaces recalled in Definition \ref{defBesov}, we conclude that
    \begin{align}\label{estimate2.heat}
     \|e^{t\Delta}W_0\|_{L^{\bq,\infty}(\R^+;L^{\tp})}\leq C\|W_0\|_{\BPS}\leq C\|n_0\|_{\BPS}.
    \end{align}
    And we get, by applying Lemma \ref{lem.convolution inequality}, that for any $T>0$, %from Corollary \ref{cor.convolution} that for any $T>0$, 
    \begin{align}\label{estimate2.A_bq,tp}
        \|\nabla A(fg)\|_{L^{\bq,\infty}_T(L^{\tp})}\leq C \|f\|_{L^{\bq,\infty}_T(L^{\tp})}\|g\|_{L^{\tq,\infty}_T(L^{\tp})},\quad
        \|\D Af\|_{L^{\bq,\infty}_T(L^{\tp})}\leq C \|f\|_{L^{\bq,\infty}_T(L^{\tp})},
    \end{align}
    where $C$ is independent of $T$.
    By substituting \eqref{estimate2.heat} and \eqref{estimate2.A_bq,tp} into \eqref{eq.W_mild solution}, we get
    \begin{align*}
        \|W\|_{L^{\bq,\infty}_T(L^{\tp})}\lesssim \|W_{0}\|_{\BPS}+\|R-W\|_{L^{\bq,\infty}_T(L^{\tp})}+\|R\|_{L^{\bq,\infty}_T(L^{\tp})}\|R\|_{L^{\tq,\infty}_T(L^{\tp})}.
    \end{align*}
    Then similar to \eqref{estimate.u_v control} and utilizing the estimate \eqref{estimate.u_L^tilde q,infty(L^tilde p)}, we ultimately derive
    \begin{align*}
    \|R\|_{L^{\bq,\infty}(\R^+;L^{\tp})}\leq C\|W\|_{L^{\bq,\infty}(\R^+;L^{\tp})}\leq C \|W_{0}\|_{\BPS}\leq C\|n_0\|_{\BPS}.
    \end{align*}
    As $\f{2}{\bq}+\s=\f{2}{\tq}$, we can further deduce
\begin{align*}
\|t^{-\f{\s}{2}}R\|_{L^{\tq,\infty}(\R^+;L^{\tp})}\leq \|t^{-\f{\s}{2}}\|_{L^{\f{2}{\s},\infty}(\R^+)}\|R\|_{L^{\bq,\infty}(\R^+;L^{\tp})}\leq C\|n_0\|_{\BPS}.
\end{align*}

    This completes the proof of Proposition \ref{prop2.u}.
\end{proof}

\begin{prop}\label{prop2.tu}
Under the assumptions of Proposition \ref{prop2.u}, we have
\begin{align}\label{estimate2.tu}
    \|tR\|_{L^{\infty}(\R^+;\dot{B}_{\tp,\infty}^{1+\f{3}{\tp}+\sigma})}+\bigl\|\bigl(tR_t,t\dot{R},(tR)_t,t\nabla^2R,t\nabla P\bigr)\bigr\|_{L^{\bq,\infty}(\R^+;L^{\tp})}\leq C\|n_0\|_{\BPS},
\end{align}
where $\dot{R}=\p_tR+R\cdot\nabla R$.
Moreover, there holds
\begin{align}\label{decay2.nabla u}
\sup_{t\in (0,\infty)}t^{1-\f{\sigma}{2}}\|\nabla R(t)\|_{L^{\infty}}\leq C\|n_0\|_{\BPS}.
\end{align}
\end{prop}

\begin{proof}
    Thanks to \eqref{Eq.linear_NS}, we have
    \begin{align*}
        (tR)_t-\Delta(tR)+\nabla(tP)=(1-\varrho)(tR)_t+\varrho R-t\varrho R\cdot\nabla R.
    \end{align*}
    Then we get, by using Lemma \ref{lem.maximal regularity}, that
    \begin{align*}
    &\|tR\|_{L^{\infty}(\R^+;\dot{B}_{\tp,\infty}^{1+\f{3}{\tp}+\sigma})}+\|((tR)_t, t\nabla^2R, t\nabla P)\|_{L^{\bq,\infty}(\R^+;L^{\tp})}\\
    & \lesssim \|(1-\varrho)(tR)_t\|_{L^{\bq,\infty}(\R^+;L^{\tp})}+\|\varrho R\|_{L^{\bq,\infty}(\R^+;L^{\tp})}+\|t\varrho R\cdot\nabla R\|_{L^{\bq,\infty}(\R^+;L^{\tp})}\\
    & \lesssim \|\varrho_0-1\|_{L^{\infty}}\|(tR)_t\|_{L^{\bq,\infty}(\R^+;L^{\tp})}+\|\varrho_0\|_{L^{\infty}} \bigl(\|R\|_{L^{\bq,\infty}(\R^+;L^{\tp})}+\|t R\cdot\nabla R\|_{L^{\bq,\infty}(\R^+;L^{\tp})}\bigr).
    \end{align*}
Thanks to \eqref{R_assumption.rho(t)}, the first term of the right hand side can be absorbed by the left hand side.
By \eqref{decay.nabla u_L^infty}, we have
    \begin{align*}
     \|t R\cdot\nabla R\|_{L^{\bq,\infty}(\R^+;L^{\tp})}\leq \bigl(t \|R(t)\|_{L^{\tp}}\|\nabla R(t)\|_{L^{\infty}}\bigr)_{L^{\bq,\infty}(\R^+)}\leq C \varepsilon_0\|R\|_{L^{\bq,\infty}(\R^+;L^{\tp})}.
    \end{align*}
    It follows from \eqref{estimate2.u} that
    \begin{align*}
    \|tR\|_{L^{\infty}(\R^+;\dot{B}_{\tp,\infty}^{1+\f{3}{\tp}+\sigma})}+\|((tR)_t,t\nabla^2R,t\nabla P)\|_{L^{\bq,\infty}(\R^+;L^{\tp})}
    \leq C\|n_0\|_{\BPS}.
    \end{align*}
    As $tR_t=(tR)_t-R$, $t\dot{R}=(tR)_t-R+tR\cdot\nabla R$, we get
    \begin{align*}
     \|(t{R}_t, t\dot{R})\|_{L^{\bq,\infty}(\R^+;L^{\tp})}\leq 2\|(tR)_t\|_{L^{\bq,\infty}(\R^+;L^{\tp})}+2\|R\|_{L^{\bq,\infty}(\R^+;L^{\tp})}+\|tR\cdot\nabla R\|_{L^{\bq,\infty}(\R^+;L^{\tp})},
    \end{align*}
    which implies \eqref{estimate2.tu}. We also have
\begin{align*}
t\|R(t)\|_{L^{\tp}}\leq \int_0^t \|(\tau R)_\tau\|_{L^{\tp}}\,\mathrm d\tau \leq Ct^{1-\f{1}{\bq}}\|n_0\|_{\BPS}= Ct^{\f12+\f{3}{2\tp}+\f{\sigma}{2}}\|n_0\|_{\BPS},
\end{align*}
and by applying the interpolation inequality, we can further deduce that
 \begin{align*}
        \|\nabla R(t)\|_{L^{\infty}}\lesssim \| R(t)\|^{1-\theta}_{L^{\tp}}\|R(t)\|^{\theta}_{\dot{B}_{\tp,\infty}^{1+\f{3}{\tp}+\sigma}}\leq Ct^{-1+\f{\sigma}{2}}\|n_0\|_{\BPS},\text{~where~} \theta=\f{\tp+3}{\tp+3+\tp \s}.
    \end{align*}

    This completes the proof of Proposition \ref{prop2.tu}.
    \end{proof}

\begin{prop}\label{prop2.tDelat dotu}
Under the assumptions of Proposition \ref{prop2.u}, we have
\begin{align}\label{estimate2.tDelta dotu}
    \|t^2\dot{R}\|_{L^{\infty}(\R^+;\dot{B}^{1+\f{3}{\tp}+\sigma}_{\tp,\infty})}+\|((t^2\dot{R})_t,t^2\nabla^2\dot{R})\|_{L^{\bq,\infty}(\R^+;L^{\tp})}\leq C\|n_0\|_{\BPS}.
\end{align}
Moreover, there holds
\begin{align}\label{estimate2.t nabla dotu}
\|t\nabla\dot{R}\|_{L^1(0,T;L^{\infty})}+\|t\dot{R}\|_{L^2(0,T;L^{\infty})}\leq CT^{\f{\sigma}{2}}\|n_0\|_{\BPS}, \quad\forall~ T>0.
\end{align}
\end{prop}

\begin{proof}
    Firstly, similar to \eqref{eq.t^2dotu} we have
    \begin{align*}
     (t^2\dot{R})_t-\Delta(t^2\dot{R})=\PP[(1-\varrho)(t^2\dot{R})_t&-t^2\varrho R\cdot\nabla\dot{R}+2t\varrho\dot{R}+t^2g]\\
     &+\Q(2t\dot{R}+2t^2R_t\cdot\nabla R)-t^2\nabla\mathrm{Tr}(\nabla R\cdot \nabla R),
    \end{align*}
    where $g=-\D R\cdot \nabla R-2\nabla R\cdot \nabla^2R+\nabla R\cdot \nabla P$.
     Then by the maximal regularity of the heat equation, we get
    \begin{align*}
    &\|t^2\dot{R}\|_{L^{\infty}(\R^+;\dot{B}_{\tp,\infty}^{1+\f{3}{p}+\s})}+\|((t^2\dot{R})_t, t^2\nabla^2\dot{R})\|_{L^{\bq,\infty}(\R^+;L^{\tp})}\\
    &\quad\leq C\|(\varrho-1)(t^2\dot{R})_t\|_{L^{\bq,\infty}(\R^+;L^{\tp})}+C\|t^2\varrho R\cdot\nabla\dot{R}-2t\varrho\dot{R}-t^2g\|_{L^{\bq,\infty}(\R^+;L^{\tp})}\\
    &\qquad+C\|t\dot{R}+t^2R_t\cdot\nabla R\|_{L^{\bq,\infty}(\R^+;L^{\tp})}+C\|t^2\nabla\mathrm{Tr}(\nabla R\cdot \nabla R)\|_{L^{\bq,\infty}(\R^+;L^{\tp})}.
\end{align*}
Thanks to \eqref{R_assumption.rho(t)}, the first term of the right hand side can be absorbed by the left hand side.
By \eqref{decay.u_L^p} and the interpolation inequality, we obtain
\begin{align*}
 &C\|t^2\varrho R\cdot\nabla\dot{R}\|_{L^{\bq,\infty}(\R^+;L^{\tp})}\leq C\|\varrho_0\|_{L^{\infty}} \bigl(t^2\|R\|_{L^{\infty}}\|\nabla\dot{R}\|_{L^{\tp}}\bigr)_{L^{\bq,\infty}(\R^+)}\\
 &\leq C \|n_0\|_{\BP}\bigl(t^{\f32}\|\dot{R}\|^{\f12}_{L^{\tp}}\|\nabla^2\dot{R}\|^{\f12}_{L^{\tp}}\bigr)_{L^{\bq,\infty}(\R^+)}\\
 &\leq C \|n_0\|_{\BP}\|t\dot{R}\|^{\f12}_{L^{\bq,\infty}(\R^+;L^{\tp})}\|t^2\nabla^2\dot{R}\|^{\f12}_{L^{\bq,\infty}(\R^+;L^{\tp})}\\
 &\leq \f14\|t^2\nabla^2\dot{R}\|_{L^{\bq,\infty}(\R^+;L^{\tp})}+C\|n_0\|^{2}_{\BP}\|t\dot{R}\|_{L^{\bq,\infty}(\R^+;L^{\tp})}.
\end{align*}
And by \eqref{decay.nabla u_L^infty}, we achieve
\begin{align*}
   &\|t^2g\|_{L^{\bq,\infty}(\R^+;L^{\tp})}+\|t^2\nabla\mathrm{Tr}(\nabla R\cdot \nabla R)\|_{L^{\bq,\infty}(\R^+;L^{\tp})}+\|t^2R_t\cdot\nabla R\|_{L^{\bq,\infty}(\R^+;L^{\tp})}\\
   &\leq C\|t\nabla R\|_{L^{\infty}(\R^+;L^{\infty})}\|(tR_t,t\nabla^2R,t\nabla P)\|_{L^{\bq,\infty}(\R^+;L^{\tp})}\leq C\varepsilon_0\|(tR_t,t\nabla^2R,t\nabla P)\|_{L^{\bq,\infty}(\R^+;L^{\tp})}.
\end{align*}
Then, putting all together with Proposition \ref{prop2.tu}, we can get \eqref{estimate2.tDelta dotu}.
Finally, by the interpolation inequality, we obtain
\begin{align}\label{estimate2.dot u infty}
   \|\nabla \dot{R}\|_{L^{\infty}}\lesssim \|\dot{R}\|^{\f12-\f{3}{2\tp}}_{L^{\tp}}\|\nabla^2\dot{R}\|^{\f12+\f{3}{2\tp}}_{L^{\tp}} \andf \|\dot{R}\|^2_{L^{\infty}}\lesssim \|\dot{R}\|^{2-\f{3}{\tp}}_{L^{\tp}}\|\nabla^2\dot{R}\|^{\f{3}{\tp}}_{L^{\tp}},
\end{align}
which gives
\begin{align*}
\int_0^T \|t\nabla\dot{R}\|_{L^{\infty}}\,\dt&\lesssim \|t\dot{R}\|^{\f12-\f{3}{2\tp}}_{L^{\bq,\infty}(0,T;L^{\tp})}\|t^2\nabla^2\dot{R}\|^{\f12+\f{3}{2\tp}}_{L^{\bq,\infty}(0,T;L^{\tp})} \|t^{-(\f12+\f{3}{2\tp})}\|_{L^{\f{\bq}{\bq-1},1}(0,T)}\\
&\leq CT^{\f{\s}{2}} \|n_0\|_{\BPS},
\end{align*}
and
\begin{align*}
\int_0^T \|t\dot{R}\|^2_{L^{\infty}}\,\dt &\lesssim \|t\dot{R}\|^{2-\f{3}{\tp}}_{L^{\bq,\infty}(0,T;L^{\tp})}\|t^2\nabla^2\dot{R}\|^{\f{3}{\tp}}_{L^{\bq,\infty}(0,T;L^{\tp})} \|t^{-\f{3}{\tp}}\|_{L^{\f{\bq}{\bq-2},1}(0,T)}
\leq CT^{\sigma} \|n_0\|^2_{\BPS}.
\end{align*}

This completes the proof of Proposition \ref{prop2.tDelat dotu}.
\end{proof}

\subsection{Proof of weak-strong uniqueness}\label{Subsec.3.3}

\begin{proof}[Proof of Theorem \ref{thm.weak-strong uniqueness}]

Let $(\rho^{(N)}, u^{(N)})$ satisfy
\begin{equation}\label{INS_N}
	\left\{
	\begin{array}{l}
		\partial_t\rho^{(N)}+u^{(N)}\cdot\nabla \rho^{(N)}=0,\\
		\rho^{(N)}(\partial_tu^{(N)}+u^{(N)}\cdot\nabla u^{(N)})-\Delta u^{(N)}+\nabla P^{(N)}=0,\\
		\dive u^{(N)} = 0,\\
		(\rho^{(N)}, u^{(N)})|_{t=0} =(\rho_{0}, u_{0,N}),
	\end{array}
	\right.
\end{equation}
where $u_{0,N}$ is determined by $\PP(\rho_0u_{0,N})=m_{0,N}$, and $m_{0,N}$ is defined in \eqref{def.rho_0u_0N}. Note that
\begin{align*}
 \|u_0-u_{0,N}\|_{L^2}&=\|\PP(u_0-u_{0,N})\|_{L^2}
 \leq\|\PP[(1-\rho_0)(u_0-u_{0,N})]\|_{L^2}+\| m_0-m_{0,N}\|_{L^2},
 \end{align*}
 which implies that (using \eqref{assumption.rho(t)})
 \begin{align*}
 \|u_0-u_{0,N}\|_{L^2}&\leq C \| m_0-m_{0,N}\|_{L^2}\leq C2^{-Ns}\|({m_0})_2\|_{\dot{B}^s_{2,\infty}}\leq C 2^{-Ns}\|m_0\|^{\f p2}_{\BP}.
\end{align*}
Then by \eqref{def.rho_0u_0N}, \eqref{m_0,1_m_0,2} and \eqref{estimate2.rho_0u_0N}, we have the following estimates
\begin{align}
&\|\rho_0-1\|_{L^{\infty}}+\|m_{0,N}\|_{\BP}\leq C\varepsilon_0,\andf\nonumber\\
 &\|m_{0,N}\|_{\BPS}\leq C\varepsilon_0^{\f{p}{\tp}}2^{N\s},\qquad
 \|u_0-u_{0,N}\|_{L^2}\leq C\varepsilon^{\f p2}_0 2^{-Ns},\label{minus.u_0}
\end{align}
where $C$ is independent of $N$.

Thanks to Theorem \ref{thm.INS_Cannone-Meyer-Planchon}, for each $N$ there exists a global solution $(\rho^{(N)}, u^{(N)})$ to \eqref{INS_N} \if 0({not need be unique})\fi, which satisfies all the estimates in Subsection \ref{Subsec.2.1} and Subsection \ref{Subsec.3.2}. %From now on, we fix this solution $(\rho^{(N)}, u^{(N)})$.
Let $(\rho,u)$ be an arbitrary weak solution provided by Lions of \eqref{INS} with initial data $(\rho_0,u_0)$.
We will show that for all $t>0$ there holds
\begin{align}\label{uniqueness2.convergence}
 (\rho^{(N)}, u^{(N)})(t)\to (\rho, u)(t) \  \text{~in~} \dot H^{-1}\times L^2\quad\text{~as~}\quad N\to \infty.
\end{align}
By the uniqueness of the limit, we deduce the uniqueness of Lions weak solutions, and also the weak-strong uniqueness.

 We denote $\delta\!\rho^{(N)}{\eqdef}\rho-\rho^{(N)}$ and $\delta\!u^{(N)}{\eqdef}u-{u}^{(N)}$. Then ($\delta\!\rho^{(N)},\delta\!u^{(N)}$) satisfies
\begin{equation}\label{Eq.delta u_N}
     \left\{
     \begin{array}{l}
     \partial_t\delta\!\rho^{(N)}+{u}^{(N)}\cdot\nabla \delta\!\rho^{(N)}+\delta\!u^{(N)}\cdot\nabla\rho=0,\qquad (t,x)\in \mathbb{R}^{+}\times\mathbb{R}^{3},\\
     \rho(\partial_t\delta\!u^{(N)}+u\cdot\nabla \delta\!u^{(N)})-\Delta \delta\!u^{(N)}+\nabla \delta\!P^{(N)}=-\delta\!\rho^{(N)}\dot{u}^{(N)}-\rho\delta\!u^{(N)}\cdot\nabla u^{(N)},\\
     \dive \delta\!u^{(N)} = 0,\\
     (\delta\!\rho^{(N)}, \delta\!u^{(N)})|_{t=0} =(0, u_0-u_{0,N}).
     \end{array}
     \right.
     \end{equation}
Here $\dot{u}^{(N)}{\eqdef}\pa_t u^{(N)}+u^{(N)}\cdot\nabla u^{(N)}$.
Let $T>0$. For all $t\in(0,T]$, set %($s$ is changed to $\tau$)
\begin{align}
&N_1(t){\eqdef}\sup_{\tau\in (0,t]}\tau^{-1}\|\delta\!\rho^{(N)}(\tau)\|_{\dot{H}^{-1}},\label{2.B_t}\\
&N_2(t){\eqdef}\|\sqrt{\rho}\delta\!u^{(N)}\|_{L^{\infty}(0,t;L^2)}+\|\nabla \delta\!u^{(N)}\|_{L^2(0,t;L^2)}.\label{2.D_t}
\end{align}
For any $\varphi\in C^{\infty}_c(\R^3)$, let $\phi$ solve\begin{align*}
\partial_\tau\phi+u^{(N)}\cdot\nabla\phi=0,\quad (\tau,x)\in(0,t]\times\R^3, \qquad \phi(t,x)=\varphi(x).
\end{align*}
 Similar to the process of \eqref{delta rho.phi}, by using \eqref{decay.nabla u_L^infty} we can also get
\begin{align*}
\|\nabla\phi(t')\|_{L^{2}}\leq C\|\nabla\varphi\|_{L^{2}}\exp\bigl(\int_{t'}^t\|\nabla u^{(N)}(\tau)\|_{L^{\infty}}\,\mathrm{d}\tau\bigr)
\leq C\|\nabla\varphi\|_{L^{2}}(t/t')^{C\varepsilon_0}, \quad \forall\ t'\in(0, t],
\end{align*}
where $C$ is a constant independent of $N$.
Then we obtain
\begin{align*}
      |\langle\delta\!\rho^{(N)}(t), \varphi\rangle|
      &\leq \int_0^t|\langle\rho\delta\!u^{(N)}(\tau),\nabla\phi(\tau)\rangle|\,\mathrm{d}\tau\leq\int_0^t\|\rho\delta\!u^{(N)}(\tau)\|_{L^2}
      \|\nabla\phi(\tau)\|_{L^{2}}\,\mathrm{d}\tau\\
     & \leq \|\rho\delta\!u^{(N)}\|_{L^{\infty}(0,t;L^2)}\|\nabla\varphi\|_{L^{2}}\int_0^t (t/\tau)^{C\varepsilon_0}\,\mathrm{d}\tau\\
     &\leq C t\|\rho_0\|_{L^{\infty}} \|\sqrt{\rho}\delta\!u^{(N)}\|_{L^{\infty}(0,t;L^2)}\|\nabla\varphi\|_{L^{2}}
    \end{align*}
    for all $t\in(0, T]$.
    Thus,  $\|\delta\!\rho^{(N)}(t)\|_{\dot{H}^{-1}}\leq CtN_2(t)$, which implies
    \begin{align}\label{eq.B<D}
     N_1(t)\leq CN_2(t).
    \end{align}
    Testing \eqref{Eq.delta u_N} against $\delta\!u^{(N)}$ gives the following energy estimate
 \begin{align}
			\sup_{t\in [0,T]}\frac12\|\sqrt\rho\delta\!u^{(N)}(t)\|_{L^2}^2+\int_{0}^T\|\nabla &\delta\!u^{(N)}(t)\|^2_{L^2}\,\dt\nonumber\\
			&\leq \|\sqrt\rho_0(u_0-u_{0,N})\|_{L^2}^2
   +2\int_{0}^T\bigl(|\mathrm I(t)|+|\mathrm{II}(t)|\bigr)\,\dt,\label{Eq.delta_u_energy}
	\end{align}
     where
    \[\mathrm I(t){\eqdef}\langle\delta\!\rho^{(N)}(t), \dot{u}^{(N)}\cdot\delta\!u^{(N)}(t)\rangle,\quad \mathrm{II}(t){\eqdef}\int_{\R^3}(\rho\delta\!u^{(N)}\cdot\nabla u^{(N)})\cdot \delta u^{(N)}(t,x)\,\dx,\quad \forall\ t\in(0, T].\]
Firstly, we have
\begin{align}\label{uniqueness2.2}
    |\mathrm{II}(t)|\leq \|\sqrt{\rho}\delta\!u^{(N)}(t)\|^2_{L^2}\|\nabla u^{(N)}(t)\|_{L^{\infty}}.
\end{align}
By duality, Sobolev embedding, Young's inequality, \eqref{2.D_t} and \eqref{eq.B<D}, we have
\begin{equation}\begin{aligned}\label{uniqueness2.1}
&\left|\int_{\R^3}\delta\!\rho^{(N)}\dot{u}^{(N)}\cdot\delta\!u^{(N)}\,\dx\right|\leq \|\delta\!\rho^{(N)}\|_{\dot{H}^{-1}}\|\dot{u}^{(N)}\cdot\delta\!u^{(N)}\|_{\dot{H}^{1}}\\
 &\leq \|\delta\!\rho^{(N)}\|_{\dot{H}^{-1}}(\|\nabla\dot{u}^{(N)}\|_{L^{\infty}}\|\delta\!u^{(N)}\|_{L^2}+\|\dot{u}^{(N)}\|_{L^{\infty}}\|\nabla\delta\!u^{(N)}\|_{L^2})\\
 &\leq \f14 \|\nabla\delta\!u^{(N)}\|^{2}_{L^2}+C \bigl(t^{-1}\|\delta\!\rho^{(N)}\|_{\dot{H}^{-1}}\|t\nabla\dot{u}^{(N)}\|_{L^{\infty}}N_2(t)+t^{-2}\|\delta\!\rho^{(N)}\|^2_{\dot{H}^{-1}}\|t\dot{u}^{(N)}\|^2_{L^{\infty}}\bigr)\\
 &\leq\f14 \|\nabla\delta\!u^{(N)}\|^{2}_{L^2}+C [N_2(t)]^2\bigl(\|t\nabla\dot{u}^{(N)}\|_{L^{\infty}}+\|t\dot{u}^{(N)}\|^2_{L^{\infty}}\bigr).
 \end{aligned}\end{equation}
Plugging \eqref{uniqueness2.1} and \eqref{uniqueness2.2} into \eqref{Eq.delta_u_energy} gives
\begin{align*}
    \sup_{t\in[0,T]}&\|\sqrt\rho\delta\!u^{(N)}(t)\|_{L^2}^2+\int_0^T\|\nabla \delta\!u^{(N)}(t)\|^2_{L^2}\,\dt\leq  2\|\sqrt\rho_0(u_0-u_{0,N})\|_{L^2}^2\\
    &+4\int_0^T\|\sqrt\rho\delta\!u^{(N)}(t)\|_{L^2}^2\|\nabla  u^{(N)}(t)\|_{L^{\infty}}\,\dt
+C\int_0^T[N_2(t)]^2\bigl(\|t\nabla\dot{u}^{(N)}\|_{L^{\infty}}+\|t\dot{u}^{(N)}\|^2_{L^{\infty}}\bigr)\,\dt.
\end{align*}
By applying Gr\"onwall's lemma, we obtain
\begin{align}\label{N2T}
N_2(T)&\leq C\|u_0-u_{0,N}\|_{L^2}\exp{\int_0^TC\bigl(\|\nabla {u}^{(N)}\|_{L^{\infty}}+\|t\nabla\dot{u}^{(N)}\|_{L^{\infty}}+\|t\dot{u}^{(N)}\|^2_{L^{\infty}})\,\dt}.
\end{align}
And by \eqref{decay.nabla u_L^infty}, \eqref{estimate2.rho_0u_0N} and \eqref{decay2.nabla u}, we can obtain
\begin{align*}%\label{decay2.nabla u_min}
    \|\nabla u^{(N)}(t)\|_{L^{\infty}}\leq C\varepsilon_0^{\f{p}{\tp}} \min\{2^{N\s}t^{-1+\f{\s}{2}},t^{-1}\},
\end{align*}
from which, we infer that\footnote{Here we consider large $N$.}
\begin{align}\label{uniqueness2.Gronwall1}
    \int_0^T \|\nabla {u}^{(N)}(t)\|_{L^{\infty}}\,\dt&\leq C\varepsilon_0^{\f{p}{\tp}}\Bigl(\int_0^{2^{-2N}} 2^{N\s}t^{-1+\f{\s}{2}}\,\dt+\int_{2^{-2N}}^T t^{-1}\,\dt\Bigr)\leq C\varepsilon_0^{\f{p}{\tp}}(\ln T+N).
\end{align}
By \eqref{estimate2.t nabla dotu} and \eqref{estimate2.rho_0u_0N}, we have
\begin{align*}
\int_0^{2^{-2N}}\|t\nabla\dot{u}^{(N)}\|_{L^{\infty}}\,\dt\leq C\varepsilon_0^{\f{p}{\tp}},\quad \int_0^{2^{-2N}}\|t\dot{u}^{(N)}\|_{L^{\infty}}^2\,\dt\leq C\varepsilon_0^{\f{2p}{\tp}}.
\end{align*}
Similar to \eqref{estimate2.dot u infty}, we have
\begin{align*}
\|\nabla \dot{u}^{(N)}\|_{L^{\infty}}\lesssim \|\dot{u}^{(N)}\|^{\f12-\f{3}{2p}}_{L^{p}}\|\nabla^2\dot{u}^{(N)}\|^{\f12+\f{3}{2p}}_{L^{p}},\quad
\|\dot{u}^{(N)}\|_{L^{\infty}}\lesssim \|\dot{u}^{(N)}\|^{1-\f{3}{2p}}_{L^{p}}\|\nabla^2\dot{u}^{(N)}\|^{\f{3}{2p}}_{L^{p}}.
\end{align*}
Then we deduce from \eqref{estimate.tu} and \eqref{estimate.t^2dotu} that
\begin{equation*}\begin{aligned}
\int_{2^{-2N}}^T\|t\nabla\dot{u}^{(N)}\|_{L^{\infty}}\,\dt
&\lesssim\|t\dot{u}^{(N)}\|^{\f12-\f{3}{2p}}_{L^{q,\infty}(0,T;L^{p})}\|t^2\nabla^2\dot{u}^{(N)}\|^{\f12+\f{3}{2p}}_{L^{q,\infty}(0,T;L^{p})} \|t^{-(\f12+\f{3}{2p})}\|_{L^{\f{q}{q-1},1}(2^{-2N},T)}\\
%\int_0^{2^{-2N}} t^{-(\f12+\f{3}{2\tp})} \|t\dot{u}^{(N)}\|^{\f12-\f{3}{2\tp}}_{L^{\tp}}\|t^2\nabla^2\dot{u}^{(N)}\|^{\f12+\f{3}{2\tp}}\,\dt\\
%&\qquad\qquad\qquad\qquad\qquad\qquad\qquad\qquad+\int_{2^{-2N}}^T t^{-(\f12+\f{3}{2p})}\|t\dot{u}^{(N)}\|^{\f12-\f{3}{2p}}_{L^{p}}\|t^2\nabla^2\dot{u}^{(N)}\|^{\f12+\f{3}{2p}}_{L^{p}}\,\dt\\
&\lesssim\varepsilon_0\|t^{\f{1}{q}-1}\|_{L^{\f{q}{q-1},1}(2^{-2N},T)}\lesssim\varepsilon_0(\ln T+N),
\end{aligned}\end{equation*}
and
\begin{equation*}\begin{aligned}
\int_{2^{-2N}}^T\|t\dot{u}^{(N)}\|^2_{L^{\infty}}\,\dt
&\lesssim\|t\dot{u}^{(N)}\|^{2-\f{3}{p}}_{L^{q,\infty}(0,T;L^{p})}\|t^2\nabla^2\dot{u}^{(N)}\|^{\f{3}{p}}_{L^{q,\infty}(0,T;L^{p})} \|t^{-\f{3}{p}}\|_{L^{\f{q}{q-2},1}(2^{-2N},T)}\\
&\lesssim\varepsilon_0^2\|t^{\f{2}{q}-1}\|_{L^{\f{q}{q-2},1}(2^{-2N},T)}\lesssim\varepsilon_0^2(\ln T+N).
\end{aligned}\end{equation*}
Thus,
\begin{align}\label{uN}
\int_0^{T}\|t\nabla\dot{u}^{(N)}\|_{L^{\infty}}\,\dt\leq C\varepsilon_0^{\f{p}{\tp}}(\ln T+N),\quad \int_0^{T}\|t\dot{u}^{(N)}\|_{L^{\infty}}^2\,\dt\leq C\varepsilon_0^{\f{2p}{\tp}}(\ln T+N).
\end{align}
So by \eqref{minus.u_0}, \eqref{N2T}, \eqref{uniqueness2.Gronwall1} and \eqref{uN},  we finally get
\begin{align}\label{Eq.uniqueness}
N_2(T)&\leq C\|u_0-u_{0,N}\|_{L^2}\cdot e^{C\varepsilon_0^{\f{p}{\tp}}(\ln T+N)}\leq CT^{C\varepsilon_0^{\f{p}{\tp}}}2^{-N(s-C\varepsilon_0^{\f{p}{\tp}})}\to 0
\end{align}
as $N\to +\infty$, which implies \eqref{uniqueness2.convergence}.
Then for any two solutions $(\rho_1,u_1)$ and $(\rho_2,u_2)$ of \eqref{INS} with the same initial data $(\rho_0,u_0)$, for all $t\in (0,T]$ we have
\begin{align*}
    &\|\rho_1(t)-\rho_2(t)\|_{\dot{H}^{-1}}\leq \|\rho_1(t)-\rho^{(N)}(t)\|_{\dot{H}^{-1}}+\|\rho^{(N)}(t)-\rho_2(t)\|_{\dot{H}^{-1}}\to 0,\\
    &\|u_1(t)-u_2(t)\|_{L^2}\leq \|u_1(t)-u^{(N)}(t)\|_{L^2}+\|u^{(N)}(t)-u_2(t)\|_{L^2}\to 0,
\end{align*}
as $N\to \infty$.
Therefore, for any solution, it coincides with the limit of $(\rho^{(N)},u^{(N)})$.

This completes the proof of Theorem \ref{thm.weak-strong uniqueness}.
\end{proof}

\section{Global well-posedness in $\dot B^{1/2}_{2,\infty}$}\label{Sec.self similar}

In this section, we prove Theorem \ref{thm.self similar}.

\subsection{A priori estimates}\label{Subsec.4.1}
In this subsection, we establish the estimates that are needed to prove Theorem \ref{thm.self similar}. Let $(\rho,u)$ be a smooth solution to the system \eqref{INS} on $[0,T^*)\times \R^3$ with $\rho_*< \rho\leq \|\rho_0\|_{L^{\infty}}$ for some constant\footnote{In this section, we require the density to be bounded away from zero to ensure the smoothness and appropriate decay of the solution $(\rho, u)$, which will be useful in some steps involving integration by parts. Nevertheless, in all estimates, the constant $C>0$ is independent of this positive lower bound $\rho_*$.} $\rho_*>0$ and $u_0\in\dot{B}^{\f12}_{2,\infty}(\R^3)$, where $T^*\in(0, +\infty]$ is the maximal time of existence.

\begin{lem}\label{lem.rho u L3}
	There exists a constant $\varepsilon_0\in(0,1)$ depending only on $\|\rho_0\|_{L^\infty}$ such that if \\ $\|u_0\|_{\BB}<\varepsilon_0$, then we have
	\begin{equation}\label{estimate3.rho u L3}
		\|\sqrt{\rho} u\|_{L^\infty(0,T^*;L^{3,\infty})}+\|t^{1/4}\nabla u\|_{L^{\infty}(0,T^*;L^2)}\leq C\|u_0\|_{\BB}
	\end{equation}
	for some constant $C>0$ depending only on $\|\rho_0\|_{L^\infty}$.
\end{lem}
\begin{proof}
For each $j\in\Z$, we consider the following coupled system of $(u_j, \nabla P_j)$:
	 \begin{equation}\label{Eq.u_j_eq}
	 	\begin{cases}
	 		\rho(\pa_tu_j+u\cdot\nabla u_j)-\Delta u_j+\nabla P_j=0,\\
	 		\dive u_j=0,\\
	 		u_j|_{t=0}=\DD_ju_0.
	 	\end{cases}
	 \end{equation}
	 Then it follows from the uniqueness of local smooth solution to \eqref{INS} that
    \begin{equation}\label{Eq.u=sum_u_j}
        u=\sum_{j\in\Z}u_j,\qquad \nabla P=\sum_{j\in\Z}\nabla P_j.
    \end{equation}

    Firstly, we can easily deduce that
    \begin{align}\label{estimate3.rho u_j L2}
    \|\sqrt{\rho}u_j\|^2_{L^{\infty}(0,T^*;L^2)}+{2}\|\nabla u_j\|^2_{L^2(0,T^*;L^2)}\leq 2\|\sqrt{\rho_0}\DD_ju_0\|^2_{L^2}\leq C 2^{-j}\|u_0\|^2_{\BB}.   \end{align}
Taking $L^2$ inner product of the momentum equation of \eqref{Eq.u_j_eq} with $\p_tu_j$, we obtain
\begin{align*}
    \f12\f{\mathrm{d}}{\dt}\|\nabla u_j(t)\|^2_{L^2}+\|\sqrt{\rho}\dot{u}_j(t)\|^2_{L^2}&=\int_{\R^3}\rho \dot{u}_j(u\cdot\nabla u_j)\,\dx\leq \|\sqrt{\rho}u\|_{L^{3,\infty}}\|\sqrt{\rho}\dot{u}_j\|_{L^2}\|\nabla u_j\|_{L^{6,2}}\\
    &\leq C \|\sqrt{\rho}u\|_{L^{3,\infty}}\|\sqrt{\rho}\dot{u}_j\|_{L^2}\|\nabla^2 u_j\|_{L^{2}},
\end{align*}
where $\dot{u}_j{\eqdef}\p_t u_j+u\cdot\nabla u_j$, and we also used the fact that\footnote{The proof of this embedding can be found in \cite{Tartar1998}, Remark 5.} $\dot{H}^1(\R^3)\hookrightarrow L^{6,2}(\R^3)$. By the Stokes estimate, we have
\begin{align*}
    \|\D u_j\|_{L^2}+\|\nabla P_j\|_{L^2}\leq C\|\sqrt{\rho}\dot{u}_j\|_{L^2}.
\end{align*}
Thus, we obtain
\begin{align*}
 \f12\f{\mathrm{d}}{\dt}\|\nabla u_j(t)\|^2_{L^2}+\|\sqrt{\rho}\dot{u}_j(t)\|^2_{L^2}\leq C \|\sqrt{\rho}u\|_{L^{3,\infty}}\|\sqrt{\rho}\dot{u}_j\|^2_{L^2}.   \end{align*}
We denote
\begin{equation}\label{def.T_*}
        T_*{\eqdef}\sup\{t\in(0, T^*]: \|\sqrt{\rho} u\|_{L^{\infty}(0, t; L^{3,\infty})}\leq c_1\}.
\end{equation}
Then we infer that for $c_1$ being so small that $Cc_1\leq \f12$,
\begin{align}\label{uj1}
 \f{\mathrm{d}}{\dt}\|\nabla u_j(t)\|^2_{L^2}+\|\sqrt{\rho}\dot{u}_j(t)\|^2_{L^2}\leq 0 \quad\text{~for~}\quad t\leq T_*,
 \end{align}
 from which, we deduce that
 \begin{align}\label{estimate3.nabla u_j L2}
 \|\nabla u_j\|_{L^\infty(0, T_*; L^2)}+\|\sqrt{\rho}\dot{u}_j\|_{L^2(0,T_*;L^2)}&\leq \|\nabla \DD_ju_0\|_{L^2}\leq C2^{j/2}\|u_0\|_{\BB}.
 \end{align}
For any $\lambda>0$, we get by \eqref{Eq.u=sum_u_j} and \eqref{Lorentz.df_Lp} that
\begin{align*}
    &\lambda^3|\{x\in\R^3:|\sqrt{\rho}u(t,x)|>\lambda\}|\\
    &\leq \lambda^3 |\{x\in\R^3:|\sum_{j>N}\sqrt{\rho}u_j(t,x)|>\lambda/2\}|+\lambda^3 |\{x\in\R^3:|\sum_{j\leq N}\sqrt{\rho}u_j(t,x)|>\lambda/2\}|\\
    &\leq \lambda^3 (2/{\lambda})^2\|\sum_{j>N}\sqrt{\rho}u_j(t,x)\|^2_{L^2}+\lambda^3(2/\lambda)^6\|\sum_{j\leq N}\sqrt{\rho}u_j(t,x)\|^6_{L^6},
\end{align*}
In view of \eqref{estimate3.rho u_j L2} and \eqref{estimate3.nabla u_j L2}, we have
\begin{align*}
\|\sum_{j>N}\sqrt{\rho}u_j(t,\cdot)\|_{L^2}\leq \sum_{j>N}\|\sqrt{\rho}u_j(t,\cdot)\|_{L^2}\leq C\sum_{j>N}2^{-j/2}\|u_0\|_{\BB}\leq C2^{-N/2}\|u_0\|_{\BB},\\
\|\sum_{j\leq N}\sqrt{\rho}u_j(t,\cdot)\|_{L^6}\leq C\sum_{j\leq N}\|\nabla u_j(t,\cdot)\|_{L^2}\leq C\sum_{j\leq N}2^{j/2}\|u_0\|_{\BB}\leq C2^{N/2}\|u_0\|_{\BB},
\end{align*}
from which, we infer that
\begin{align*}
    \lambda^3|\{x\in\R^3:|\sqrt{\rho}u(t,x)|>\lambda\}|\leq C\lambda2^{-N}\|u_0\|^2_{\BB}+C\lambda^{-3}2^{3N}\|u_0\|^6_{\BB}.
\end{align*}
Hence, for any $\lambda>0$, taking $N=[{\log_2}(\lambda\|u_0\|^{-1}_{\BB})]+1$, we finally get
\begin{align*}
    \|\sqrt{\rho}u\|_{L^{\infty}(0,T_*;L^{3,\infty})}=\sup_{\lambda>0} \lambda|\{x\in\R^3:|\sqrt{\rho}u(t,x)|>\lambda\}|^{1/3}
    \leq C\|u_0\|_{\BB} \quad\text{~for~} \quad t\leq T_*.
\end{align*}
Taking $\varepsilon_0$ so small that $C\|u_0\|_{\BB}\leq C\varepsilon_0\leq \f{c_1}{2}$ for $c_1$ given by \eqref{def.T_*}, we infer that $T_*=T^*$ and this in turn shows the first estimate of \eqref{estimate3.rho u L3}.

By \eqref{estimate3.rho u_j L2}, we know that  for any $0<t<T^*$, there exists $\sigma_1\in(0,t)$ such that
\begin{align*}
    \|\nabla u_j(\sigma_1)\|^2_{L^2}\leq t^{-1}\|\sqrt{\rho_0}\DD_ju_0\|^2_{L^2},
\end{align*}
which along with \eqref{uj1} gives
\begin{align*}
    \|\nabla u_j(t)\|^2_{L^2}+\int_{\sigma_1}^t\|\sqrt{\rho}\dot{u}_j(s)\|^2_{L^2}\,\ds\leq \|\nabla u_j(\sigma_1)\|^2_{L^2}\leq t^{-1}\|\sqrt{\rho_0}\DD_ju_0\|^2_{L^2}.
\end{align*}
Therefore, we finally infer that for any $0<t<T^*$,
\begin{align}\label{estimate3.nabla u_j min}
    \|\nabla u_j(t)\|^2_{L^2}\leq \min \bigl\{t^{-1}\|\sqrt{\rho_0}\DD_ju_0\|^2_{L^2}, ~\|\nabla \DD_ju_0\|^2_{L^2}\bigr\}.
\end{align}

For any fixed $0<t<T^*$, by \eqref{Eq.u=sum_u_j} and \eqref{estimate3.nabla u_j min}, we have
\begin{align*}
    \|\nabla u(t)\|^2_{L^2}&\leq \sum_{j\leq N}\|\nabla u_j(t)\|^2_{L^2}+\sum_{j>N}\|\nabla u_j(t)\|^2_{L^2}
    \leq \sum_{j\leq N}\|\nabla\DD_j u_0\|^2_{L^2}+\sum_{j>N}t^{-1}\|\sqrt{\rho_0}\DD_ju_0\|^2_{L^2}\\
    &\leq C2^N\|u_0\|^2_{\BB}+Ct^{-1}2^{-N}\|u_0\|^2_{\BB}.
\end{align*}
Taking $N=1+[{\log_2} (t^{-1/2})]$, we finally get the second estimate of \eqref{estimate3.rho u L3}.
\end{proof}

\begin{lem}\label{lem.nabla u_T,2T}
	Let $\varepsilon_0\in(0,1)$ be given by Lemma \ref{lem.rho u L3}. If $\|u_0\|_{\BB}<\varepsilon_0$, then for any $0<T_1\leq T_2<T^*$, we have
	\begin{align}
		&\sup_{t\in[T_1,T_2]}t^{1/2}\|\nabla u(t)\|^2_{L^2}+\|t^{1/4}(\sqrt{\rho} \dot{u},\nabla^2u,\nabla P)\|^2_{L^2(T_1,T_2;L^{2})}\leq C\bigl(\ln(T_2/T_1)+1\bigr)\|u_0\|^2_{\BB},\label{estimate3.rho dotu}\\
  &\|\sqrt{\rho} \dot{u}(\sigma_2)\|_{L^2}\leq CT_1^{-3/4}\|u_0\|_{\BB},\quad \exists~\sigma_2\in [T_1/2,T_1],\label{decay3.rho dotu}
	\end{align}
	for some constant $C>0$ depending only on $\|\rho_0\|_{L^\infty}$.
\end{lem}

\begin{proof}
   By taking $L^2$ inner product of the momentum equation of \eqref{INS} with $\p_t u$, we obtain
	\begin{align*}
		\frac12\frac{\mathrm d}{\dt}\|\nabla u(t)\|_{L^2}^2+\|\sqrt\rho\dot u(t)\|_{L^2}^2&=\int_{\R^3}\rho \dot{u}(u\cdot \nabla u)\,\dx\leq \|\sqrt\rho\dot u(t)\|_{L^2}\|\sqrt{\rho}u\|_{L^{3,\infty}}\|\nabla u\|_{L^{6,2}}\\
  &\leq C\|\sqrt{\rho}u\|_{L^{3,\infty}}\|\sqrt\rho\dot u\|_{L^2}\|\nabla^2u\|_{L^2},
    \end{align*}
    where we have used the embedding $\dot{H}^1(\R^3)\hookrightarrow L^{6,2}(\R^3)$. As we also have the following Stokes estimate
   \begin{align*}
    \|\D u\|_{L^2}+\|\nabla P\|_{L^2}\leq C\|\sqrt{\rho}\dot{u}\|_{L^2},
   \end{align*}
and  then
  \begin{align*}
  \frac12\frac{\mathrm d}{\dt}\|\nabla u(t)\|_{L^2}^2+\|\sqrt\rho\dot u(t)\|_{L^2}^2
  &\leq C\|\sqrt{\rho}u\|_{L^{3,\infty}}\|\sqrt\rho\dot u\|^2_{L^2},
	\end{align*}
 which along with Lemma \ref{lem.rho u L3} shows
 \begin{align*}
& \frac{\mathrm d}{\dt}\|\nabla u(t)\|_{L^2}^2+\|\sqrt\rho\dot u(t)\|_{L^2}^2\leq 0,\\
 &\frac{\mathrm d}{\dt}\bigl(t^{1/2}\|\nabla u(t)\|_{L^2}^2\bigr)+t^{1/2}\|\sqrt\rho\dot u(t)\|_{L^2}^2\leq \f12t^{-1/2}\|\nabla u(t)\|_{L^2}^2.
 \end{align*}
Thus, for any $0<\delta\leq T_1\leq T_2<T^*$, we get by \eqref{estimate3.rho u L3}  that
 \begin{equation}\begin{aligned}
\sup_{t\in[\delta,T_2]}t^{1/2}\|\nabla u(t)\|^2_{L^2}+\|t^{1/4}\sqrt{\rho} \dot{u}\|^2_{L^2(\delta,T_2;L^{2})}&\leq \delta^{1/2}\|\nabla u(\delta)\|^2_{L^2}+\f12\|t^{-1/4}\nabla u\|^2_{L^2(\delta,T_2;L^{2})}\\
&\leq C\bigl(\ln(T_2/\delta)+1\bigr)\|u_0\|^2_{\BB}.\label{estimate3.sigma2}
\end{aligned}\end{equation}
By \eqref{estimate3.sigma2} we know that
\begin{align*}
    \int_{T_1/2}^{T_1}t^{1/2}\|\sqrt{\rho}\dot{u}\|^2_{L^2}\,\dt\leq C\|u_0\|^2_{\BB},
\end{align*}
which implies \eqref{decay3.rho dotu}.
\end{proof}

\begin{lem}\label{lem.nabla u L_infty}
	Let $\varepsilon_0\in(0,1)$ be given by Lemma \ref{lem.rho u L3}.
  If $\|u_0\|_{\BB}<\varepsilon_0$, then for any $0<T_1<T^*/2$, we have
	\begin{align}
		&\sup_{t\in [T_1,2T_1]}t^{3/2}\|\sqrt{\rho}\dot{u}\|^2_{L^2}+\|t^{3/4}\nabla\dot{u}\|^2_{L^2(T_1,2T_1;L^2)}\leq C\|u_0\|^2_{\BB},\label{Eq.low_order_est_rhou}\\
  &\andf
   \int_{T_1}^{2T_1} \|\nabla u(t)\|_{L^{\infty}}\,\dt\leq C\|u_0\|_{\BB},\label{L1_T(Lip)}
	\end{align}
	for some constant $C>0$ depending only on $\|\rho_0\|_{L^{\infty}}$.
\end{lem}

\begin{proof}
    Applying the material derivative $D_t$ to the momentum equation of \eqref{INS}, we get
	\begin{equation*}
			\rho(\pa_t\dot{u}+u\cdot\nabla \dot{u})-\Delta\dot{u}
			=-\pa_k(\pa_ku\cdot \nabla u)-\pa_k u\cdot \nabla \pa_ku
			-(\nabla P_t+u\cdot \nabla (\nabla P)).
	\end{equation*}
	Taking $L^2$ inner product with $\dot{u}$ of the above equation, we get by integration by parts and $\dive u=0$ that
	\begin{equation*}%\label{Eq.Dt_energy}
			\frac12 \frac{\mathrm d}{\dt}\|\sqrt\rho\dot u\|_{L^2}^2+\|\nabla \dot{u}(t)\|^2_{L^2}
			=\int_{\R^3}(\pa_ku\cdot \nabla u)\cdot\pa_k\dot{u}\,\dx+\int_{\R^3}\pa_ku\cdot(\pa_ku\cdot\nabla\dot{u})\,\dx+J,
	\end{equation*}
	where
	$$J(t)\eqdef -\int_{\R^3}(\nabla P_t+u\cdot \nabla (\nabla P))\cdot\dot{u}\,\dx,\quad\forall\ t\in[0, T^*).$$
	Along the same lines as in the proof of Lemma 3.2 in \cite{HSWZ1} and Lemma 3.5 in \cite{HSWZ2}, we have
	\begin{equation*}
		J(t)=\frac{\mathrm d}{\dt}\int_{\R^3} P\pa_iu^j\pa_ju^i\,\dx
		-3\int_{\R^3} P \pa_iu^j\pa_j\dot{u}^i\,\dx+2\int_{\R^3} P\pa_iu^k\pa_ku^j\pa_ju^i\,\dx.
	\end{equation*}
	for all $t\in[0, T^*)$.  Let
	\begin{equation*}
		\Psi(t)\eqdef\|\sqrt\rho\dot u(t)\|_{L^2}^2-2\int_{\R^3} P\pa_iu^j\pa_ju^i\,\dx,\quad\forall\ t\in[0, T^*).
	\end{equation*}
	Then we have {(see (3.25) and (3.26) in \cite{HSWZ2})}\if0
	\begin{align*}
		&\frac12\Psi'(t)+\|\nabla \dot{u}(t)\|^2_{L^2}
=\int_{\R^3}(\pa_ku\cdot \nabla u)\cdot\pa_k\dot{u}\,\dx+\int_{\R^3}\pa_ku\cdot(\pa_ku\cdot\nabla\dot{u})\,\dx\\
&\qquad\qquad-3\int_{\R^3} P \pa_iu^j\pa_j\dot{u}^i\,\dx+2\int_{\R^3} P\pa_iu^k\pa_ku^j\pa_ju^i\,\dx\\
		\leq&\ C\|\nabla u(t)\|_{L^6}\|\nabla u(t)\|_{L^3}\|\nabla \dot u(t)\|_{L^2}+C\|P(t)\|_{L^6}\|\nabla u(t)\|_{L^3}\|\nabla \dot u(t)\|_{L^2}\\&+C\|P(t)\|_{L^6}\|\nabla u(t)\|_{L^3}^2\|\nabla u(t)\|_{L^6}\\
		\leq&\ C\|\nabla u(t)\|_{L^3}\|\nabla \dot u(t)\|_{L^2}\|(\nabla^2u(t), \nabla P(t))\|_{L^2}
        +C\|\nabla u(t)\|_{L^3}^2\|(\nabla^2u(t), \nabla P(t))\|_{L^2}^2\\
		\leq&C\|\nabla u(t)\|_{L^3}\|\nabla \dot u(t)\|_{L^2}\|\sqrt\rho\dot u(t)\|_{L^2}+C\|\nabla u(t)\|_{L^3}^2\|\sqrt\rho\dot u(t)\|_{L^2}^2\\
		\leq&\ \frac12\|\nabla \dot{u}(t)\|^2_{L^2}+C\|\nabla u(t)\|_{L^3}^2\|\sqrt\rho\dot u(t)\|_{L^2}^2,
		\end{align*}
		hence,\fi
  \begin{align}
      \Psi'(t)&+\|\nabla \dot{u}(t)\|^2_{L^2}\leq C\|\nabla u(t)\|_{L^3}^2\|\sqrt\rho\dot u(t)\|_{L^2}^2,\nonumber\\
      (t^{3/2}\Psi)'(t)&+t^{3/2}\|\nabla \dot{u}(t)\|^2_{L^2}\leq \f32 t^{1/2}\Psi(t)+Ct^{3/2}\|\nabla u(t)\|_{L^3}^2\|\sqrt\rho\dot u(t)\|_{L^2}^2\label{Eq.Psi'(t)}
  \end{align}
  for all $t\in[0, T^*)$, and\if0

        On the other hand, we have
		\begin{align*}
			&\left|\int_{\R^3} P\pa_iu^j\pa_ju^i\,\dx\right|\leq \|P(t)\|_{L^6}\|\nabla u(t)\|_{L^3}\|\nabla u(t)\|_{L^2}\\
			&\leq C\|\nabla P(t)\|_{L^2}\|\nabla u(t)\|_{L^2}^{3/2}\|\nabla^2u(t)\|_{L^2}^{1/2}
			\leq C\|\nabla u(t)\|_{L^2}^{3/2}\|\sqrt\rho\dot u(t)\|_{L^2}^{3/2}\\
            &\leq \f12 \|\sqrt\rho\dot u(t)\|_{L^2}^2+C\|\nabla u(t)\|_{L^2}^6,
		\end{align*}
		thus,\fi
		\begin{equation}\label{Eq.Psi_est}
			\frac12\|\sqrt\rho\dot u(t)\|_{L^2}^2-C\|\nabla u(t)\|_{L^2}^6\leq \Psi(t)\leq 2\|\sqrt\rho\dot u(t)\|_{L^2}^2+C\|\nabla u(t)\|_{L^2}^6, \quad\forall\ t\in[0, T^*).
		\end{equation}
   It follows from \eqref{Eq.Psi'(t)} and \eqref{Eq.Psi_est} that
    \begin{align*}
        &(t^{3/2}\Psi)'(t)+t^{3/2}\|\nabla \dot{u}(t)\|^2_{L^2}\\
        &\qquad\leq Ct^{1/2}\|\sqrt\rho\dot u(t)\|_{L^2}^2+Ct^{1/2}\|\nabla u(t)\|_{L^2}^6
        +Ct^{3/2}\|\nabla u(t)\|_{L^3}^2\|\sqrt\rho\dot u(t)\|_{L^2}^2,\quad\forall\ t\in[0, T^*).
    \end{align*}
    Then for the $\sigma_2\in (T_1/2,T_1)$ defined in \eqref{decay3.rho dotu}, we have
    \begin{align}\label{Eq.t^3/2Psi}
        &\sup_{t\in [\sigma_2,2T_1]} t^{3/2}\Psi(t)+\int_{\sigma_2}^{2T_1}t^{3/2}\|\nabla \dot{u}(t)\|^2_{L^2}\,\dt
    \leq C\sigma_2^{3/2}\|\sqrt\rho\dot u(\sigma_2)\|_{L^2}^2\\
    &\quad+C\int_{\sigma_2}^{2T_1}t^{1/2}\|\sqrt\rho\dot u(t)\|_{L^2}^2\,\dt\left(1+\sup_{\sigma_2<t<2T_1}t\|\nabla u(t)\|_{L^3}^2\right)
    +C\int_{\sigma_2}^{2T_1}t^{1/2}\|\nabla u(t)\|_{L^2}^6\,\dt.\nonumber
    \end{align}
     Moreover, it follows from \eqref{estimate3.rho u L3}, \eqref{estimate3.rho dotu} and {\eqref{decay3.rho dotu}} that
    \begin{align*}
        \sigma_2^{3/2}\|\sqrt\rho\dot u(\sigma_2)\|_{L^2}^2+\int_{\sigma_2}^{2T_1}t^{1/2}&\|\sqrt\rho\dot u(t)\|_{L^2}^2\,\dt+
        \int_{\sigma_2}^{2T_1}t^{1/2}\|\nabla u(t)\|_{L^2}^6\,\dt
        \leq C\|u_0\|^2_{\BB},
    \end{align*}
    and we also have (for all $t>0$, also using \eqref{estimate3.rho u L3})
    \begin{align*}
        t\|\nabla u(t)\|_{L^3}^2&\leq Ct^{1/4}\|\nabla u(t)\|_{L^2}t^{3/4}\|\nabla^2u(t)\|_{L^2}\leq C\|u_0\|_{\BB}t^{3/4}\|\sqrt\rho\dot u(t)\|_{L^2}.
    \end{align*}
    Thus, \eqref{Eq.t^3/2Psi} implies that (also using $\|u_0\|_{\BB}<\varepsilon_0<1$)
    \begin{equation}\label{3.30}
        \sup_{t\in (\sigma_2,2T_1)}t^{3/2}\Psi(t)+\int_{\sigma_2}^{2T_1}t^{3/2}\|\nabla \dot{u}(t)\|^2_{L^2}\,\dt\leq C\|u_0\|_{\BB}^2(1+A),
    \end{equation}
    where $A\eqdef\sup_{\sigma_2<s<2T_1}s^{3/4}\|\sqrt\rho\dot u(s)\|_{L^2}$. Thus, by \eqref{Eq.Psi_est}, \eqref{3.30} and  \eqref{estimate3.rho u L3}, for any $\sigma_2<t\leq 2T_1$, we obtain
    \begin{align*}
        &t^{3/2}\|\sqrt\rho\dot u(t)\|_{L^2}^2+\int_{\sigma_2}^ts^{3/2}\|\nabla \dot{u}(s)\|^2_{L^2}\,\ds
        \leq C\|u_0\|_{\BB}^2(1+A)+Ct^{3/2}\|\nabla u(t)\|_{L^2}^6\\
        &\leq C\|u_0\|_{\BB}^2(1+A)+C\|t^{1/4}\nabla u\|_{L^\infty(0,T^*; L^2)}^6
        \leq C\|u_0\|_{\BB}^2(1+A)+C\|u_0\|_{\BB}^6,\\
        &A^2+\int_{\sigma_2}^{2T_1}s^{3/2}\|\nabla \dot{u}(s)\|^2_{L^2}\,\ds\leq C\|u_0\|_{\BB}^2(1+A)+C\|u_0\|_{\BB}^6,\\
       &A^2+\int_{\sigma_2}^{2T_1}s^{3/2}\|\nabla \dot{u}(s)\|^2_{L^2}\,\ds\leq C(\|u_0\|_{\BB}^2+\|u_0\|_{\BB}^4)+C\|u_0\|_{\BB}^6\leq  C\|u_0\|_{\BB}^2.
\end{align*}
This proves \eqref{Eq.low_order_est_rhou}. Note that
\begin{align*}
    \|\nabla {u}(t)\|_{L^{\infty}}
    \leq C\|\nabla {u}(t)\|^{1/2}_{L^{6}}\|\nabla^2 {u}(t)\|^{1/2}_{L^{6}}\leq C\|\sqrt{\rho}\dot{u}(t)\|^{1/2}_{L^{2}}\|\nabla\dot{u}(t)\|^{1/2}_{L^{2}},
\end{align*}
which along with \eqref{estimate3.rho dotu} and \eqref{Eq.low_order_est_rhou} shows
\begin{align*}
   \int_{T_1}^{2T_1}\|\nabla {u}(t)\|_{L^{\infty}}\,\dt&\leq C \|t^{-1/2}\|_{L^2(T_1,2T_1)}\|t^{1/4}\sqrt{\rho}\dot{u}\|^{1/2}_{L^2(T_1,2T_1;L^2)}\|t^{3/4}\nabla\dot{u}(t)\|^{1/2}_{L^2(T_1,2T_1;L^2)}\\
   &\leq C\|u_0\|_{\BB}.
\end{align*}

    This completes the proof of Lemma \ref{lem.nabla u L_infty}.
\end{proof}

\begin{cor}[Estimates for $u_t$]
    Let $\varepsilon_0\in(0,1)$ be given by Lemma \ref{lem.rho u L3}. If $\|u_0\|_{\BB}<\varepsilon_0$, then for any $0<T_1<T^*/2$,
 we have
 \begin{equation}\begin{aligned}\label{Eq.high_order_est_ut}
		\sup_{t\in [T_1,2T_1]}t^{3/2}\|\sqrt\rho u_t\|^2_{L^2}+\|t^{1/4}&\sqrt\rho u_t\|^2_{L^2(T_1,2T_1; L^2)}\\
        &\qquad\qquad\qquad+\|t^{3/4}\nabla u_t\|^2_{L^2(T_1,2T_1;L^2)}\leq C\|u_0\|^2_{\BB},
\end{aligned}\end{equation}
	for some constant $C>0$ depending only on $\|\rho_0\|_{L^{\infty}}$.
\end{cor}
\begin{proof}
Using the inequality $\|f\|_{L^\infty(\R^3)}\leq C\|\nabla f\|_{L^2(\R^3)}^{1/2}\|\nabla f\|_{L^6(\R^3)}^{1/2}$, we have
    \begin{align*}
     &\int_{\R^3}\rho |u\cdot\nabla u|^2\,\dx\leq \|\sqrt{\rho} u(t)\|^2_{L^{3,\infty}}\|\nabla u(t)\|^2_{L^{6,2}}\leq C\|\sqrt{\rho} u(t)\|^2_{L^{3,\infty}} \|\nabla^2 u(t)\|_{L^2}^{2};\\
     &\|\nabla(u\cdot\nabla u)(t)\|_{L^2}\leq \|\nabla u(t)\|_{L^3}\|\nabla u(t)\|_{L^6}+\|u(t)\|_{L^\infty}\|\nabla^2 u(t)\|_{L^2}
     \leq C\|\nabla u(t)\|_{L^2}^{\f12}\|\nabla^2u(t)\|_{L^2}^{\f32}.
    \end{align*}
By \eqref{estimate3.rho u L3}, \eqref{estimate3.rho dotu}, \eqref{Eq.low_order_est_rhou} and the Stokes estimate, we infer that for any $t\in [T_1,2T_1]$,
    \begin{align*}
        t^{3/2}\|\sqrt\rho u\cdot\nabla u(t)\|_{L^2}^2
        &\leq Ct^{3/2}\|\sqrt{\rho} u(t)\|^2_{L^{3,\infty}} \|\nabla^2 u(t)\|_{L^2}^{2}\\
        &\leq Ct^{3/2}\|\sqrt{\rho} u(t)\|^2_{L^{3,\infty}} \|\sqrt{\rho}\dot{u}\|_{L^2}^{2}\leq C\|u_0\|_{\BB}^4,\\
        \int_{T_1}^{2T_1}t^{1/2}\|\sqrt\rho u\cdot\nabla u(t)\|_{L^2}^2\,\dt&\leq \|\sqrt{\rho} u\|^2_{L^{\infty}(0,T^*;L^{3,\infty})} \int_{T_1}^{2T_1}t^{1/2}\|\sqrt{\rho}\dot{u}\|_{L^2}^{2}\,\dt
        \leq C\|u_0\|_{\BB}^4.
    \end{align*}
        And we can also obtain by using the Stokes estimate that
        \begin{align*}
        &\int_{T_1}^{2T_1}t^{3/2}\|\nabla(u\cdot\nabla u)(t)\|_{L^2}^2\,\dt\leq C\int_{T_1}^{2T_1}t^{3/2}\|\nabla u(t)\|_{L^2}\|\nabla^2u(t)\|_{L^2}^3\,\dt\\
        &\leq C\|t^{1/4}\nabla u\|_{L^\infty(0,T^*; L^2)}\|t^{3/4}\sqrt\rho\dot u\|_{L^\infty(T_1,2T_1; L^2)}\int_{T_1}^{2T_1}t^{1/2}\|\sqrt\rho\dot u(t)\|_{L^2}^2\,\dt\leq C\|u_0\|_{\BB}^4.
    \end{align*}
    Therefore, \eqref{Eq.high_order_est_ut} follows from \eqref{estimate3.rho dotu}, \eqref{Eq.low_order_est_rhou} and $u_t=\dot u-u\cdot\nabla u$.
\end{proof}

\begin{cor}[Estimates for $\sqrt{\rho_0}u_t$]
    Let $\varepsilon_0\in(0,1)$ be given by Lemma \ref{lem.rho u L3}. If $\|u_0\|_{\BB}<\varepsilon_0$, then for any $0<T<2T<T^*$
	\begin{equation}\label{Eq.rho_0ut_est}
	\int_{T}^{2T}\|\sqrt{\rho_0} u_t\|^2_{L^2}\,\dt\leq CT^{-1/2}\|u_0\|^2_{\BB}
	\end{equation}
	for some constant $C>0$ depending only on $\|\rho_0\|_{L^{\infty}}$.
\end{cor}
\begin{proof}
By the density equation of \eqref{INS} and \eqref{estimate3.rho u L3}, we get\footnote{The definition of $\dot{W}^{-1,(3,\infty)}$ can be found in Appendix \ref{Appen_proof}, Definition \ref{def.W-1,3,infty}.}
\begin{align}\label{Eq.rho_W-13}
\|\rho(t)-\rho_0\|_{\dot{W}^{-1,(3,\infty)}}\leq \int_0^t \|\rho u(s)\|_{L^{3,\infty}}\,\ds\leq t\|\rho u\|_{L^{\infty}(0,t;L^{3,\infty})}\leq Ct\|u_0\|_{\BB}.
\end{align}
Note that
\begin{align*}
\int_{T}^{2T}\int_{\R^3} \rho_0 |u_t|^2\,\dx\,\dt=\int_{T}^{2T}\int_{\R^3} \rho(t) |u_t|^2\,\dx\,\dt-\int_{T}^{2T}\int_{\R^3} (\rho(t)-\rho_0) |u_t|^2\,\dx\,\dt.
\end{align*}
By duality and \eqref{Eq.rho_W-13}, we obtain
\begin{align*}
&\int_{T}^{2T}\Bigl|\int_{\R^3} (\rho(t)-\rho_0) |u_t|^2\,\dx\Bigr|\,\dt\leq \int_{T}^{2T}\|\rho(t)-\rho_0\|_{\dot{W}^{-1,(3,\infty)}}\||u_t|^2\|_{\dot{W}^{1,(\f32,1)}}\,\dt\\
&\leq 2\int_{T}^{2T}\|\rho(t)-\rho_0\|_{\dot{W}^{-1,(3,\infty)}}\|\nabla u_t\|_{L^2}\|u_t\|_{L^{6,2}}\,\dt
\leq C\|u_0\|_{\BB}\int_{T}^{2T}t\|\nabla u_t\|^2_{L^2}\,\dt.
\end{align*}
Therefore, \eqref{Eq.rho_0ut_est} follows from \eqref{Eq.high_order_est_ut}.
\end{proof}

\subsection{Existence part of Theorem \ref{thm.self similar}}\label{Subsec.4.2}

This subsection is devoted to the proof of global existence part of the solution in Theorem \ref{thm.self similar}.
This is similar to Subsection \ref{Subsec.2.2}, and here we also focus on the proof of \eqref{Eq.convergence}.
\if0
We consider (where $\varepsilon\in(0,1)$)
$$ u_0^{\varepsilon}\in C^{\infty}(\R^3)
\andf \rho_0^{\varepsilon}\in C^{\infty}(\R^3) $$
such that
\begin{equation*}
u_0^{\varepsilon}\to u_0 \text{~in~} \BB(\R^3), \quad \rho_0^{\varepsilon}\rightharpoonup \rho_0 \text{~in~} L^{\infty} \text{~weak-*},
\andf~ \rho_0^{\varepsilon}\to \rho_0 \text{~in~} L^{p}_{\text{loc}}(\R^3) \text{~if~} p<\infty,
\end{equation*}
as $\varepsilon\to 0+$.
In light of the classical strong solution theory for the system \eqref{INS}, there exists a unique global smooth solution $(\rho^{\varepsilon}, u^{\varepsilon}, P^{\varepsilon})$
corresponding to data $(\rho_0^{\varepsilon},u_0^{\varepsilon})$. Thus, the triple $(\rho^{\varepsilon}, u^{\varepsilon}, P^{\varepsilon})$  satisfies all the a priori estimates of Subsection \ref{Subsec.4.1} uniformly with respect to $\varepsilon\in(0,1)$. Hence, $(\rho^{\varepsilon}, u^{\varepsilon})$ converges weakly (or weakly-*) to a limit $(\rho, u)$ as $\varepsilon\to0+$, up to subsequence. To show that the limit solves \eqref{INS} weakly, in view of standard compactness arguments, it suffices to prove that

\begin{equation}\label{3Eq.convergence}
    \lim_{\varepsilon\to0+}\int_0^t\langle \rho^\varepsilon u^\varepsilon\otimes u^\varepsilon(s)-\rho u\otimes u(s), \nabla\varphi\rangle\,\ds=0\quad\forall\ t>0
\end{equation}
for any divergence-free function $\varphi\in C_c^\infty([0, +\infty)\times\R^3)$.
\fi
 Since
\begin{align*}
    \int_0^t|\langle \rho^\varepsilon u^\varepsilon\otimes u^\varepsilon(s),\nabla\varphi\rangle|\,\ds
    &\leq\|\sqrt{\rho^{\varepsilon}}u^\varepsilon\|^2_{L^{\infty}(0,t;L^{3,\infty})}\|\nabla \varphi\|_{L^{1}(0,t;L^{3,1})}\leq Ct,\\
    \int_0^t|\langle \rho u\otimes u(s), \nabla\varphi\rangle|\,\ds&\leq \|\sqrt{\rho}u\|^2_{L^{\infty}(0,t;L^{3,\infty})}\|\nabla \varphi\|_{L^{1}(0,t;L^{3,1})}\leq Ct,
\end{align*}
for all $\varepsilon\in(0,1)$, $t>0$, where $C>0$ is a constant depending on $\|\rho_0\|_{L^{\infty}},\|u_0\|_{\BB}$.
Now we fix $t>0$. Let $\eta>0$. Taking $\delta_0{\eqdef}\min\{t/2, \eta/(4C)\}>0$, we have
\[\int_0^{\delta_0}\left|\langle \rho^\varepsilon u^\varepsilon\otimes u^\varepsilon(s)-\rho u\otimes u(s), \nabla\varphi\rangle\right|\,\ds<\frac\eta2,\quad\forall\ \varepsilon\in(0,1).\]
To prove \eqref{Eq.convergence}, as \eqref{Eq.convergence_delta1} can be also established similarly, it suffices to show \eqref{Eq.convergence_delta2}.

Thanks to \eqref{Eq.high_order_est_ut}, we infer that $\{t^{3/4}\p_tu^{\varepsilon}\}$ is uniformly bounded in $L^{2}(\delta_0,t;L^6)$, from which we know $\{\p_tu^{\varepsilon}\}$ is uniformly bounded in $L^{2}(\delta_0,t;L^6)$.
By \eqref{estimate3.rho dotu} and \eqref{estimate3.rho u L3}, we get that $t^{1/4}u^{\varepsilon}$ is uniformly bounded in $L^{2}(\delta_0,t;\dot{H}^1\cap\dot{H}^2)$, which implies  that $u^{\varepsilon}$ is uniformly bounded in $L^{2}(\delta_0,t;\dot{H}^1\cap\dot{H}^2)$. Then by the Ascoli-Arzela theorem, we conclude that
\begin{align*}
    &u^{\varepsilon}\to u \text{~in~} L^{2}_{loc}(\delta_0,t;L^6_{loc}),\\
     &\nabla u^{\varepsilon}\to \nabla u \text{~in~} L^{2}_{loc}(\delta_0,t;L^2_{loc}),
\end{align*}
 which along with the uniform boundedness of $\rho^\varepsilon $ in $L^{\infty}$ implies
\begin{equation*}
    \lim_{\varepsilon\to0+}\int_{\delta_0}^t\left|\langle \rho^\varepsilon (u^\varepsilon\otimes u^\varepsilon-u\otimes u)(s), \nabla\varphi\rangle\right|\,\ds=0.
\end{equation*}

\subsection{Uniqueness part of Theorem \ref{thm.self similar}}\label{Subsec.4.3}
The goal of this subsection is to prove the uniqueness part of Theorem \ref{thm.self similar}. Indeed, it will be a consequence of the following proposition.
\begin{prop}\label{uni.prop}
   Let $T>0$. Let $(\rho,u,\nabla P)$ and $(\bar\rho,\bar u, \nabla \bar P)$ be two solutions of \eqref{INS} on $[0,T]\times\R^3$ corresponding to the same initial data. Let $C>0$ be a constant depending only on $\|\rho_0\|_{L^{\infty}}$. In addition, we assume that there exists a small quantity $\varepsilon\in (0,1/2)$ such that for any $0<s<t\leq T$, the following conditions hold
\begin{itemize}
    \item $\|\sqrt{\rho}(\bar u-u)\|_{L^{\infty}(0,t;L^2)}\leq Ct^{1/4}$;
    \item $\|\nabla \bar u-\nabla u\|_{L^2(0,t;L^2)}\leq Ct^{1/4}$;
    \item $\|\nabla \bar u\|_{L^1(s,t;L^{\infty})}\leq C\varepsilon \bigl(\ln(t/s)+1\bigr)$;
    \item $\|\nabla \bar u\|_{L^4(s,t;L^2)}\leq C\varepsilon [\ln(t/s)]^{1/4}$;
    \item $\|\tau^{3/4}\nabla \dot{\bar u}(\tau)\|_{L^2(s,t;L^2)}\leq C\varepsilon \bigl([\ln(t/s)]^{1/2}+1\bigr)$, where $\dot{\bar{u}}{\eqdef}\partial_t\bar{u}+\bar{u}\cdot\nabla \bar{u}$.
\end{itemize}
Then $(\rho,u,\nabla P)=(\bar\rho,\bar u, \nabla\bar P)$ on $[0,T]\times \R^3$.
\end{prop}

\begin{proof}
Denote $\delta\!\rho{\eqdef}\rho-\bar{\rho}$ and $\delta\!u{\eqdef}u-\bar{u}$. Then ($\delta\!\rho,\delta\!u$) satisfies \eqref{Eq.delta u}.
Let $T>0$. For all $t\in(0,T]$, we set
\begin{align}
&A(t){\eqdef}\sup_{s\in (0,t]}s^{-\f34}\|\delta\!\rho(s)\|_{\dot{W}^{-1,(3,\infty)}},\label{A_t}\\
&B(t){\eqdef}\Bigl(\sup_{s\in [0,t]}\|\sqrt{\rho}\delta\!u(s)\|^2_{L^2}+\|\nabla \delta\!u\|^2_{L^2(0,t;L^2)}\Bigr)^{1/2}.\label{B_t}
\end{align}
Based on our assumptions, we derive the crucial estimate
\begin{align}\label{B_leq_t1/4}
    B(t)\leq Ct^{1/4}, \quad \forall~ 0<t\leq T.
\end{align}

For any $\varphi\in C^{\infty}_c(\R^3)$, by the classical theory on transport equations (\cite{BCD}, Theorem 3.2), there exists a unique $\phi(s,x)$ solving
\begin{align*}
\partial_s\phi+\bar{u}\cdot\nabla\phi=0,\quad (s,x)\in(0,t]\times\R^3, \qquad \phi(t,x)=\varphi(x),
\end{align*}
 with the estimate
\begin{align}\label{uniqueness3.nabla varphi}
\|\nabla\phi(s)\|_{L^{\f32,1}}\leq C\|\nabla\varphi\|_{L^{\f32,1}}\exp \int_s^t\|\nabla \bar{u}(\tau)\|_{L^{\infty}}\,\mathrm{d}\tau,\quad\forall\ s\in(0, t].
\end{align}
Note that
\begin{align}\label{uniqueness3.exp.inte}
 \int_0^t  e^{C\varepsilon\bigl(\ln(t/s)+1\bigr)} \,\ds=te^{C\varepsilon}\int_0^1 e^{C\varepsilon|\ln(\tau)|} \,\mathrm{d}\tau\leq Ct\int_0^1\tau^{-C\varepsilon}\,\mathrm{d}\tau\leq Ct.
\end{align}
Similar to \eqref{delta rho.phi}, we have
\begin{align*}
\frac{\mathrm{d}}{\ds} \langle\delta\!\rho(s), \phi(s)\rangle
&=\langle -\dive\{ \delta\!\rho\bar{u}\},\phi\rangle-\langle \dive \{\rho \delta\!u\},\phi\rangle-\langle \delta\!\rho,\bar{u}\cdot\nabla\phi\rangle
=\langle \rho\delta\!u,\nabla\phi\rangle.
\end{align*}
Integrating over $(0,t)$, and using \eqref{uniqueness3.nabla varphi} and \eqref{uniqueness3.exp.inte} yield that
\begin{align}
|\langle\delta\!\rho(t), \phi(t)\rangle|&\leq
\int_0^t |\langle \rho \delta\!u,\nabla\phi\rangle|\,\ds \leq \int_0^t \|\rho\delta\!u(s)\|_{L^{3,\infty}}\|\nabla\phi(s)\|_{L^{\f32,1}}\,\ds\nonumber\\
&\leq Ct^{\f34}\|\nabla\varphi\|_{L^{\f32,1}}\|\rho\delta\!u\|_{L^4(0,t;L^{3,\infty})}\label{b}.
\end{align}
By the interpolation inequality and Sobolev embedding, we have
\begin{align*}
\int_0^t\|\rho\delta\!u\|^4_{L^{3,\infty}}\,\ds \leq \int_0^t\|\rho\delta\!u\|^2_{L^2}\|\rho\delta\!u\|^2_{L^6}\,\ds\leq C\|\sqrt{\rho}\delta\!u\|^2_{L^{\infty}(0,t;L^2)}\|\nabla \delta\!u\|^2_{L^2(0,t;L^2)}.
\end{align*}
Then it follows from \eqref{A_t}, \eqref{B_t} and \eqref{b} that
\begin{align}\label{Eq.A<B}
  A(t)\leq CB(t),\quad\forall\ t\in(0, T].
\end{align}

Next, testing \eqref{Eq.delta u} against $\delta\!u$ gives  the following energy estimate
 \begin{align}
			\sup_{t\in[\delta,T]}\frac12 \|\sqrt\rho\delta\!u(t)\|_{L^2}^2&+\int_{\delta}^T\|\nabla \delta\!u(t)\|^2_{L^2}\,\dt\nonumber\\
			&\leq  \|\sqrt\rho\delta\!u(\delta)\|_{L^2}^2
+2\int_{\delta}^T\Bigl|\int_{\R^3}\delta\!\rho\dot{\bar{u}}\cdot\delta\!u\,\dx\Bigr|\,\dt+2\int_{\delta}^T\int_{\R^3}\bigl|(\rho\delta\!u\cdot\nabla\bar{u})\cdot\delta\!u\bigr|\,\dx\,\dt.\label{uniquness3.ener}
	\end{align}
By H\"older's inequality and Sobolev embedding, we have
\begin{equation}\begin{aligned}\label{uniqueness3.1}
 \left|\int_{\R^3}(\rho\delta\!u\cdot\nabla\bar{u})\cdot\delta\!u\,\dx\right|&\leq \|\rho\delta\!u\|_{L^{3,\infty}}\|\nabla\bar{u}\|_{L^2}\|\delta\!u\|_{L^{6,2}}\leq C\|\sqrt{\rho}\delta\!u\|^{\f12}_{L^2}\|\nabla\delta\!u\|^{\f32}_{L^2}\|\nabla\bar{u}\|_{L^2}\\
 &\leq \f18 \|\nabla\delta\!u\|^{2}_{L^2}+C \|\sqrt{\rho}\delta\!u\|^{2}_{L^2}\|\nabla\bar{u}\|^4_{L^2}.
\end{aligned}\end{equation}
By duality, Sobolev embedding and Young's inequality, we have
\begin{align}
&\left|\int_{\R^3}\delta\!\rho\dot{\bar{u}}\cdot\delta\!u\,\dx\right|\leq \|\delta\!\rho\|_{\dot{W}^{-1,(3,\infty)}}\|\dot{\bar{u}}\cdot\delta\!u\|_{\dot{W}^{1,(3/2,1)}}\nonumber\\
 &\leq \|\delta\!\rho\|_{\dot{W}^{-1,(3,\infty)}}(\|\nabla\dot{\bar{u}}\|_{L^2}\|\delta\!u\|_{L^{6,2}}+\|\dot{\bar{u}}\|_{L^{6,2}}\|\nabla\delta\!u\|_{L^2})\leq C\|\delta\!\rho\|_{\dot{W}^{-1,(3,\infty)}}\|\nabla\dot{\bar{u}}\|_{L^2}\|\nabla\delta\!u\|_{L^2}\label{uniqueness3.2}\\
 &\leq \f18 \|\nabla\delta\!u\|^{2}_{L^2}+C t^{-\f32}\|\delta\!\rho(t)\|^2_{\dot{W}^{-1,(3,\infty)}} \cdot t^{\f32}\|\nabla\dot{\bar{u}}\|^2_{L^2}\leq\f18 \|\nabla\delta\!u(t)\|^{2}_{L^2}+C A^2(t)t^{\f32}\|\nabla\dot{\bar{u}}(t)\|^2_{L^2}.\nonumber
 \end{align}
Hence, plugging \eqref{Eq.A<B}, \eqref{uniqueness3.1} and \eqref{uniqueness3.2} into \eqref{uniquness3.ener} gives
\begin{equation}\begin{aligned}\label{3.delta u}
    B^2(T)&\leq 2B^2(\delta)+C\int_{\delta}^T\|\sqrt\rho\delta\!u(t)\|_{L^2}^2\|\nabla \bar u(t)\|_{L^2}^4\,\dt+C\int_{\delta}^T A^2(t)t^{3/2}\|\nabla\dot{\bar{u}}(t)\|^2_{L^2}\,\dt\\
    &\leq 2B^2(\delta)+ C\int_{\delta}^T B^2(t)\bigl(\|\nabla \bar u(t)\|_{L^2}^4+t^{3/2}\|\nabla\dot{\bar{u}}(t)\|^2_{L^2}\bigr)\,\dt\\
    &\leq 2B^2(\delta)+ C\int_{\delta}^T B^2(t)\g(t)\,\dt,
\end{aligned}\end{equation}
where $0\leq\g(t){\eqdef}\|\nabla \bar u(t)\|_{L^2}^4+t^{3/2}\|\nabla\dot{\bar u}(t)\|_{L^2}^2$ and
\begin{align*}
\int_{\delta}^T\|\nabla \bar u(t)\|_{L^2}^4\,\dt\leq C\varepsilon^4\ln({T}/{\delta}) \andf \int_{\delta}^Tt^{3/2}\|\nabla\dot{\bar u}(t)\|_{L^2}^2\,\dt\leq C\varepsilon^2 \bigl(\ln({T}/{\delta})+1\bigr)
\end{align*}
for any $0<\delta<T$.
Applying Gr\"onwall's lemma and by \eqref{3.delta u}, \eqref{B_leq_t1/4} we finally get
\begin{equation*}
    B^2(T)\leq 2B^2(\delta)\exp \bigl(C\int_{\delta}^T\g(t)\,\dt\bigr)
    \leq C\delta^{1/2}(T/\delta)^{C\varepsilon^2}\to 0 \quad \text{~as~} \,\delta\to 0,
\end{equation*}
which implies that $B(t)\equiv0$ on $[0,T]$, then by \eqref{Eq.A<B} we know that $A(t)\equiv0$ on $[0,T]$.

This completes the proof of Proposition \ref{uni.prop}.
\end{proof}

We claim that the above proposition implies the uniqueness part of Theorem \ref{thm.self similar}.
Firstly, by \eqref{Eq.rho_W-13} we know that
$$\|\rho(t)-\rho_0\|_{\dot{W}^{-1,(3,\infty)}}+\|\bar\rho(t)-\rho_0\|_{\dot{W}^{-1,(3,\infty)}}\leq Ct\|u_0\|_{\BB},$$
thus $t^{-\f34}\delta\!\rho$ belongs to $L^{\infty}(0,T;\dot{W}^{-1,(3,\infty)})$.
Next we show that $\delta\!u$ not only lies in the energy space, but also exhibits a $t^{1/4}$ bound. Note that
\begin{equation}\begin{aligned}\label{Eq.u-u_0_L^2}
\int_{\R^3} \rho(t,x)|u(t,x)-\bar u(t,x)|^2\,\dx=&\int_{\R^3}\rho_0(x)|u(t,x)-\bar u(t,x)|^2\,\dx\\
&+\int_{\R^3}(\rho(t,x)-\rho_0(x)) |u(t,x)-\bar u(t,x)|^2\,\dx.
\end{aligned}\end{equation}
By \eqref{Eq.rho_0ut_est}, we have
\begin{equation}\begin{aligned}\label{Eq.u-u_0_rho_0}
&\|\sqrt{\rho_0}(u(t)-u_0)\|_{L^2(\R^3)}\leq C\int_0^t \|\sqrt{\rho_0}u_t(s)\|_{L^2}\,\ds=C\sum_{n=0}^{\infty}\int_{2^{-(n+1)}t}^{2^{-n}t} \|\sqrt{\rho_0}u_t(s)\|_{L^2}\,\ds\\
&\leq C\sum_{n=0}^{\infty}(2^{-n}t-2^{-(n+1)}t)^{1/2}\Bigl(\int_{2^{-(n+1)}t}^{2^{-n}t} \|\sqrt{\rho_0}u_t(s)\|^2_{L^2}\,\ds\Bigr)^{1/2}\\
&\leq Ct^{\f14}\|u_0\|_{\BB}\sum_{n=0}^{\infty}2^{-\f{n+1}{4}}\leq Ct^{\f14}\|u_0\|_{\BB},
\end{aligned}\end{equation}
and the same inequality holds for $\bar u$. On the other hand, by duality, Sobolev embedding, \eqref{Eq.rho_W-13}, \eqref{Eq.low_order_est_rhou} and \eqref{estimate3.rho u L3}, we have
\begin{align}
&\Bigl|\int_{\R^3}(\rho(t)-\rho_0) |u(t)-\bar u(t)|^2\,\dx\Bigr| \leq \|\rho(t)-\rho_0\|_{\dot{W}^{-1,(3,\infty)}}\|\nabla (|u(t)-\bar u(t)|^2)\|_{L^{3/2,1}}\nonumber\\
&\leq 2\|\rho(t)-\rho_0\|_{\dot{W}^{-1,(3,\infty)}}\|\nabla u(t)-\nabla\bar u(t)\|_{L^{2}}\|u(t)-\bar u(t)\|_{L^{6,2}} \nonumber\\
&\leq Ct\|\nabla u(t)-\nabla \bar u(t)\|^2_{L^{2}}\|u_0\|_{\BB}\nonumber\\
&\leq Ct^{1/2} \|t^{1/4}(\nabla u,\nabla\bar u)\|^2_{L^{\infty}(\R^+;L^{2})}\|u_0\|_{\BB}
\leq Ct^{1/2}\|u_0\|^3_{\BB}.\label{Eq.u-u_0_duality}
\end{align}
And by \eqref{estimate3.rho u L3}, we also have
\begin{align}
\int_0^t\|(\nabla u,\nabla\bar u)(s)\|^2_{L^2}\,\ds\leq C\|u_0\|^{2}_{\BB}\int_0^t s^{-1/2}\,\ds\leq Ct^{\f12}\|u_0\|^{2}_{\BB},\label{estimate3.nabla delta u decay}\\
\int_s^t\|\nabla\bar u(\tau)\|^4_{L^2}\,\mathrm{d}\tau\leq C\|u_0\|^{4}_{\BB}\int_s^t \tau^{-1}\,\mathrm{d}\tau\leq C\ln(t/s)\|u_0\|^{4}_{\BB}.\nonumber
\end{align}
Plugging \eqref{Eq.u-u_0_rho_0} and \eqref{Eq.u-u_0_duality} into \eqref{Eq.u-u_0_L^2} and \eqref{estimate3.nabla delta u decay} gives
\begin{align*}%\label{estimate3.delta u t1/4}
\|\sqrt\rho(u-\bar u)\|^2_{L^{\infty}(0,t;L^2)}+ \|\nabla u-\nabla\bar{u}\|^2_{L^2(0,t;L^2)}\leq Ct^{1/2}\|u_0\|^2_{\BB}.
\end{align*}
For $0<s<t\leq T$, let $N=[\log_2(t/s)]+1$, by %Lemma \ref{lem.rho u L3} and
 Lemma \ref{lem.nabla u L_infty}, we have
\begin{align*}%\label{estimate3.delta u t1/4}
\int_s^t\|\nabla\bar u(\tau)\|_{L^{\infty}}\,\mathrm{d}\tau&\leq \sum_{n=0}^{N}\int_{2^{-(n+1)}t}^{2^{-n}t} \|\nabla\bar u(\tau)\|_{L^{\infty}}\,\mathrm{d}\tau\leq \sum_{n=0}^{N}C\|u_0\|_{\BB}\\&=C(N+1)\|u_0\|_{\BB}\leq C(\ln(t/s)+1)\|u_0\|_{\BB}.
\end{align*}
Similarly, $\|\tau^{3/4}\nabla \dot{\bar u}(\tau)\|_{L^2(s,t;L^2)}\leq C(N+1)^{1/2}\|u_0\|_{\BB}\leq C\|u_0\|_{\BB} \bigl([\ln(t/s)]^{1/2}+1\bigr).$ Thus, all the assumptions of Proposition \ref{uni.prop} are satisfied by the solutions obtained in Theorem \ref{thm.self similar}, and then we complete the proof of Theorem \ref{thm.self similar}.

\appendix

\section{Littlewood-Paley and Lorentz spaces}\label{Appen_proof}
In this section, we shall collect  the functional spaces used in this paper and some related lemmas. Let us first recall the following dyadic operators:
\begin{equation}\begin{split}\label{defparaproduct}
&\dot{\Delta}_ju\eqdef \mathcal{F}^{-1}\bigl(\varphi(2^{-j}|\xi|)\widehat{u}\bigr),
 \quad\
\dot{S}_ju\eqdef\mathcal{F}^{-1}\bigl(\chi(2^{-j}|\xi|)\widehat{u}\bigr),
\end{split}\end{equation}
where $\xi=(\xi_1,\xi_2,\xi_3)$, $\mathcal{F} u$ and
$\widehat{u}$ denote the Fourier transform of $u$,
while $\mathcal{F}^{-1} u$ denotes its inverse,
$\chi(\tau)$ and $\varphi(\tau)$ are smooth functions such that
\begin{align*}
&\Supp \varphi \subset \Bigl\{\tau \in \R\,: \, \frac34 \leq
|\tau| \leq \frac83 \Bigr\}\quad\mbox{and}\quad \forall
 \tau>0\,,\ \sum_{j\in\Z}\varphi(2^{-j}\tau)=1,\\
&\Supp \chi \subset \Bigl\{\tau \in \R\,: \, |\tau| \leq
\frac43 \Bigr\}\quad\mbox{and}\quad \forall
 \tau\in\R\,,\ \chi(\tau)+ \sum_{j\geq 0}\varphi(2^{-j}\tau)=1.
\end{align*}

By using these dyadic operators, we can define the homogeneous Besov spaces as follows.

\begin{defi}\label{defBesov}
{\sl Let $p,r\in[1,\infty]$ and $s\in\R$. The homogeneous Besov space $\dot{B}^s_{p,r}$ consists of those $u\in{\mathcal S}'$ with $\|\dot{S}_ju\|_{L^\infty}\rightarrow0$ as $j\rightarrow-\infty$ such that
$$\|u\|_{\dot{B}^s_{p,r}}\eqdef\big\|\big(2^{js}
\|\dot{\Delta}_j u\|_{L^p}\big)_{j\in\Z}\bigr\|
_{\ell ^{r}(\Z)}<\infty.$$
Moreover, if $s<0$,   one has the equivalent norm
$$\|u\|_{\dot{B}^s_{p,r}}\sim
\|t^{-\f{s}{2}}\|e^{t\D}u\|_{L^p}\bigr\|_{L^r(\R^+,\f{\mathrm{d}t}{t})}.$$
}\end{defi}

\if0
Here we give the equivalent definition of negative indices Besov spaces.

\begin{defi}[Theorem 2.34, \cite{BCD}]\label{def.negative Besov}
Let $s$ be a positive real number and $(p,r)\in [1,\infty]^2$. A constant $C$ exists which satisfies
\begin{align*}
    C^{-1}\|u\|_{\dot{B}^{-2s}_{p,r}}\leq \bigl\|\|t^se^{t\D}u\|_{L^p}\bigr\|_{L^r(\R^+,\f{dt}{t})}\leq C\|u\|_{\dot{B}^{-2s}_{p,r}}.
\end{align*}
\end{defi}
\fi

Next we define the Lorentz space, see \cite{Gra}.
\begin{defi}\label{defLorentz}
{\sl Let $p,r\in[1,\infty]$. The Lorentz space $L^{p,r}(X, \mu)$ on a measure space $(X, \mu)$ consists of all measurable functions $f$ with
\begin{equation*}
\|f\|_{L^{p,r}(X, \mu)}\eqdef
\left\{
     \begin{aligned}
&\Bigl(\int_0^{\infty}\bigl(t^{\f1p}f^{*}(t)\bigr)^{r}
\frac{\dt}{t}\Bigr)^{\f1r}
\quad\text{ if } 1\leq r<\infty,\\
&\sup_{t>0} t^{\f1p}f^{*}(t)
\quad\qquad\qquad\text{ if }  r=\infty
     \end{aligned}
     \right.
\end{equation*}
finite, where
$$f^{*}(t)\eqdef \inf \bigl\{ s\geq0:
d_f(s)\leq t\bigr\}\andf
d_f(s)\eqdef\mu\{x\in X:|f(x)|>s\}.$$
}\end{defi}

\begin{rmk}
If $p<\infty$, the Lorentz norm can be equivalently given by
\begin{equation*}
\|f\|_{L^{p,r}}\eqdef
\left\{
     \begin{aligned}
&p^{\f1r}\Bigl(\int_0^{\infty}
\bigl(d_f^{\f1p}(t)t\bigr)^{r}
\frac{\dt}{t}\Bigr)^{\f1r}
\quad\text{ if } 1\leq r<\infty,\\
&\sup_{t>0} d_f^{\f1p}(t) t
~\qquad\qquad\qquad\text{ if }  r=\infty.
     \end{aligned}
     \right.
\end{equation*}
And we can easily get
\begin{align}\label{Lorentz.df_Lp}
 d_f(s) \leq s^{-p} \|f\|^p_{L^{p,\infty}}\leq s^{-p} \|f\|^p_{L^p}, \quad\forall~ s>0,~ f\in L^p.
\end{align}
\end{rmk}

\begin{defi}\label{def.W-1,3,infty}
We define $\dot{W}^{1,(3/2,1)}$ to be completion of $C_c^\infty(\R^3)$ under the norm $\|\cdot\|_{\dot{W}^{1,(3/2,1)}}$ given by
\[\|f\|_{\dot{W}^{1,(3/2,1)}}{\eqdef}\sum_{|\alpha|=1}\|\p^{\alpha}f\|_{L^{3/2,1}}.\]
And we denote the dual space of $\dot{W}^{1,(3/2,1)}$ by $\dot{W}^{-1,(3,\infty)}(\R^3)$. More precisely, $\dot{W}^{-1,(3,\infty)}(\R^3)$ consists of all distributions $T\in\mathcal D'(\R^3)$ such that there exists a constant $C>0$ satisfying
\[|\langle T,f\rangle|\leq C\|f\|_{\dot{W}^{1,(3/2,1)}},\quad\forall\ f\in C_c^\infty(\R^3),\]
with the norm
\[\|T\|_{\dot{W}^{-1,(3,\infty)}(\R^3)}{\eqdef}\sup \{|\langle T,f\rangle|: f\in C_c^\infty(\R^3)\ \text{with}\ \|f\|_{ \dot{W}^{1,(3/2,1)}}\leq1\}.\]
\end{defi}

The following maximal regularity will play an essential role in this paper.

\begin{lem}[Proposition A.5 in \cite{DW}]\label{lem.maximal regularity}
{\sl Let $1<p,~q<\infty$ and $1\leq r \leq \infty$. Then for any $u_0\in \dot{B}^{2-\frac{2}{q}}_{p,r}$ with $\dive u_0=0$, and any $f\in L^{q,r}(0,T;L^{p}(\R^d))$, the following {Stokes} system
\begin{equation*}%\label{eq.stokes system}
     \left\{
     \begin{array}{l}
     \partial_tu-\Delta u+\nabla P=f, \quad (t,x)\in \R^{+}\times\R^d\\
     \dive u=0,\\
     u|_{t=0}=u_0,
     \end{array}
     \right.
\end{equation*}
has a unique solution $(u,\nabla P)$ with $\nabla P\in L^{q,r}(0,T;L^p(\R^d))$, and there exists a constant $C$ independent of $T$ such that
\begin{equation*}\begin{split}
\|u\|_{L^{\infty}(0,T;\dot{B}^{2-\f{2}{q}}_{p,r})}
+\|u_t, \nabla^2u,\nabla P\|_{L^{q,r}(0,T;L^p)}
\leq C \big(\|u_0\|_{\dot{B}^{2-\f{2}{q}}_{p,r}}
+\|f\|_{L^{q,r}(0,T;L^p)}\big).
\end{split}\end{equation*}
}\end{lem}

\section* {Acknowledgments}

D. Wei is partially supported by the National Key R\&D Program of China under the grant 2021YFA1001500.  P. Zhang is partially  supported by National Key R$\&$D Program of China under grant 2021YFA1000800 and by National Natural Science Foundation of China under Grant 12421001 and 12288201.
 Z. Zhang is partially supported by NSF of China under Grant 12288101.


\begin{thebibliography}{99}

\bibitem{Abidi2007} H. Abidi, \'Equation de Navier-Stokes avec densit\'e et viscosit\'e variables dans l'espace critique. {\it Rev. Mat. Iberoam.}, {\bf 23} (2007), 537-586.

\bibitem{Abidi_Gui2021} H. Abidi and G. Gui, Global well-posedness for the 2-D inhomogeneous incompressible Navier-Stokes system with large initial data in critical spaces. {\it Arch. Ration. Mech. Anal.}, {\bf 242} (2021), 1533-1570.

\bibitem{Abidi_Gui_Zhang2012} H. Abidi, G. Gui and P. Zhang, On the well-posedness of three-dimensional inhomogeneous Navier-Stokes equations in the critical spaces. {\it Arch. Ration. Mech. Anal.}, {\bf 204} (2012), 189-230.

\bibitem{Abidi_Gui_Zhang2013} H. Abidi, G. Gui and P. Zhang, Well-posedness of 3-D inhomogeneous Navier-Stokes equations with highly oscillatory initial velocity field. {\it J. Math. Pures Appl. (9)}, {\bf 100} (2013), 166-203.

\bibitem{Abidi_Gui_Zhang2023} H. Abidi, G. Gui and P. Zhang, On the global existence and uniqueness of solution to 2-D inhomogeneous incompressible Navier-Stokes equations in critical spaces. {\it J. Differential Equations}, {\bf 406} (2024), 126-173.

\bibitem{Abidi-Paicu2007} H. Abidi and M. Paicu, Existence globale pour un fluide inhomog$\grave{e}$ne. {\it Ann. Inst. Fourier (Grenoble)}, {\bf 57} (2007), 883-917.

\bibitem{BCD} H. Bahouri, J. Y. Chemin and R. Danchin, {\it Fourier Analysis and Nonlinear Partial Differential Equations}, Grundlehren der Mathematischen Wissenschaften, vol. {\bf343}, 2011, Springer-Verlag, Berlin, Heidelberg.

\bibitem{Barker2017} T. Barker, Existence and weak* stability for the Navier-Stokes system with initial values in critical Besov spaces. arXiv:1703.06841.

\bibitem{Barker2018} T. Barker, Uniqueness results for weak Leray-Hopf solutions of the Navier-Stokes system with initial values in critical spaces. {\it J. Math. Fluid Mech.}, {\bf20} (2018), 133-160.

\bibitem{Cannone-Meyer-Planchon} M. Cannone, Y. Meyer and F. Planchon, Solutions auto-similaires des \'equations de Navier-Stokes. {\it S\'eminaire sur les \'Equations aux D\'eriv\'ees Partielles}, (1993-1994), Exp. No. VIII, 12 pp.

\bibitem{Chemin2011} J. Y. Chemin, About weak-strong uniqueness for the $3$D incompressible Navier-Stokes system. {\it Comm. Pure Appl. Math.}, {\bf64} (2011), 1587-1598.

\bibitem{CZZ} D. Chen, Z. Zhang and W. Zhao, Fujita-Kato theorem for the 3-D inhomogeneous Navier-Stokes equations. {\it J. Differential Equations}, {\bf261}(2016), 738-761.

\bibitem{CSV} T. Crin-Barat, S. \v Skondri\'c and A. Violini, Relative energy method for weak-strong uniqueness of the inhomogeneous Navier-Stokes equations. arXiv:2404.12858.

\bibitem{Danchin2003} R. Danchin, Density-dependent incompressible viscous fluids in critical spaces. {\it Proc. Roy. Soc. Edinburgh Sect. A}, {\bf 133} (2003), 1311-1334.

\bibitem{Danchin2024} R. Danchin, Global well-posedness for 2D inhomogeneous viscous flows with rough data via dynamic interpolation. To appear in {\it Anal. PDE}; arXiv:2404.02541.

\bibitem{Danchin_Mucha2012} R. Danchin and P. B. Mucha, A Lagrangian approach for the incompressible Navier-Stokes equations with variable density. {\it Comm. Pure Appl. Math.}, {\bf 65} (2012), 1458-1480.

\bibitem{DM3} R. Danchin and P. B. Mucha, Incompressible flows with piecewise constant density. {\it Arch. Ration. Mech. Anal.}, {\bf207} (2013), 991-1023.

\bibitem{DMT} R. Danchin, P. B. Mucha and P. Tolksdorf, Lorentz spaces in action on pressureless systems arising from models of {collective} behavior.  {\it J. Evol. Equ.}, {\bf21} (2021), 3103-3127.

\bibitem{DV} R. Danchin and I. Vasilyev, Density-dependent incompressible Navier-Stokes equations in critical tent spaces. arXiv:2305.09027.

\bibitem{DW} R. Danchin and S. Wang, Global unique solutions for the inhomogeneous Navier-Stokes equations with only bounded density in critical regularity spaces. {\it Comm. Math. Phys.}, {\bf399} (2023), 1647-1688.

\bibitem{DZ} B. Dong and Z. Zhang, On the weak-strong uniqueness of Koch-Tataru's solution for the Navier-Stokes equations. {\it J. Differential Equations}, {\bf256} (2014), 2406-2422.

%\bibitem{FJR1972} E. B. Fabes, B. F. Jones and N. M. Riviere, The initial value problem for the Navier-Stokes equations with data in $L^p$, {\it Arch. Ration. Mech. Anal.}, {\bf45} (1972), 222-240.

\bibitem{Fujita-Kato} H. Fujita and T. Kato, On the Navier-Stokes initial value problem I.  {\it Arch. Ration. Mech. Anal.}, {\bf16} (1964), 269-315.

%\bibitem{Gallagher-Planchon2002} I. Gallagher and F. Planchon, On global infinite energy solutions to the Navier-Stokes equations in two dimensions, {\it Arch. Ration. Mech. Anal.}, {\bf161} (2002), no. 4, 307–337.

\bibitem{Germain2006} P. Germain, On the Navier-Stokes initial value problem I.  {\it J. Differential Equations}, {\bf226} (2006), no. 2, 373-428.

\bibitem{Gra} L. Grafakos,
{\it Classical Fourier analysis. Second edition.} Graduate Texts in Mathematics, {\bf249}. Springer, New York, 2008. xvi+489 pp.

\bibitem{HSWZ1} T. Hao, F. Shao, D. Wei and Z. Zhang, On the density patch problem for the 2-D inhomogeneous Navier-Stokes equations. arXiv:2406.07984.

\bibitem{HSWZ2} T. Hao, F. Shao, D. Wei and Z. Zhang, Global well-posedness of inhomogeneous Navier-Stokes equations with bounded density. arXiv:2406.19907.

\bibitem{Kato} T. Kato, Strong $L^p$-solutions of the Navier-Stokes equation in $\R^m$, with applications to weak solutions. {\it Math. Z.}, {\bf187} (1984), 471-480.

\bibitem{K} A. V. Kazhikhov, Solvability of the initial-boundary value problem for the equations of the motion of an inhomogeneous viscous incompressible fluid.  {\it Dokl. Akad. Nauk SSSR}, {\bf216} (1974), 1008-1010.

\bibitem{Koch-Tataru} H. Koch and D. Tataru, Well-posedness for the Navier-Stokes equations. {\it Adv. Math.}, {\bf157} (2001), 22-35.

%\bibitem{LS} O. A. Lady\v{z}enskaja and  V. A. Solonnikov, The unique solvability of an initial-boundary value problem for viscous incompressible inhomogeneous fluids. (Russian) Boundary value problems of mathematical physics, and related questions of the theory of functions, 8, {\it Zap. Nau\v{c}n. Sem. Leningrad. Otdel. Mat. Inst. Steklov.}, {\bf52} (1975), 52-109, 218-219.

\bibitem{Leray1934} J. Leray, Sur le mouvement d'un liquide visqueux emplissant l'espace. {\it Acta Math.}, {\bf63} (1934), 183-248.

\bibitem{Lions_vol1} P. L. Lions, {\it Mathematical topics in fluid mechanics. Vol. 1. Incompressible models}, Oxford Lecture Ser. Math. Appl., {\bf 3}. Oxford Sci. Publ. The Clarendon Press, Oxford University Press, New York, 1996. xiv+237 pp.

\bibitem{ONeil} R. O'Neil, Convolution operators and $L(p,q)$ spaces.  {\it Duke Math. J.}, {\bf30} (1963), 129-142.

\bibitem{Paicu_zhang2012} M. Paicu and P. Zhang, Global solutions to the 3-D incompressible inhomogeneous Navier-Stokes system. {\it J. Funct. Anal.}, {\bf 262} (2012), 3556-3584.

\bibitem{Paicu_Zhang_Zhang_CPDE} M. Paicu, P. Zhang and Z. Zhang, Global unique solvability of inhomogeneous Navier-Stokes equations with bounded density. {\it Comm. Partial Differential Equations}, {\bf 38} (2013), 1208-1234.

\bibitem{S} J. Simon, Nonhomogeneous viscous incompressible fluids: existence of velocity, density, and pressure. {\it SIAM J. Math. Anal.}, {\bf21} (1990), 1093-1117.

\bibitem{Tartar1998} L. Tartar, Imbedding theorems of Sobolev spaces into Lorentz spaces. {\it Bollettino dell'Unione Matematica Italiana}, Serie 8, 1-B (3) (1998), 479-500.

\bibitem{Zhang_Adv} P. Zhang, Global Fujita-Kato solution of 3-D inhomogeneous incompressible Navier-Stokes system. {\it Adv. Math.}, {\bf 363} (2020), 107007, 43 pp.


\end{thebibliography}
\end{document}